\documentclass[12pt]{amsart}

\usepackage{amssymb,latexsym}
\usepackage{graphicx}
\usepackage{stmaryrd}
\usepackage[pdfusetitle]{hyperref}
\usepackage{color}
\usepackage{soul,xcolor}
\usepackage{chngcntr}

\def \diam {{\rm diam}}

\def \dist {{\rm dist}}
\def \rdist{{\rm rdist}}

\def\cal{\mathcal}
\def \rdist {{\rm rdist}}
\def \ecc {{\rm ecc}}
\def \ec {{\rm ecc}}
\def \BMO {{\rm BMO}}
\def \CMO {{\rm CMO}}
\def\R{{\mathbb R}}
\def\S{{\mathcal S}}

\def \sign {{\rm sign}}

\newcommand{\N}{\mathbb{N}}
\def\N{\mathbb{N}}
\newcommand{\supp}{\mathrm{supp \;}}

\newcommand{\norm}[1]{\|#1\|}
\newcommand{\abs}[1]{|#1|}
\newcommand{\inner}[2]{\left\langle #1,#2 \right\rangle}

\newtheorem{theorem}{Theorem}[section]
\newtheorem{definition}[theorem]{Definition}
\newtheorem{lemma}[theorem]{Lemma}

\newtheorem{proposition}[theorem]{Proposition}

\newtheorem{example}[theorem]{Example}
\newtheorem{remark}[theorem]{Remark}

\address{Centre for Mathematical Sciences, University of Lund, Lund, Sweden}
\email{janfreol@maths.lth.se}
\email{paco.villarroya@maths.lth.se}

\author{Jan-Fredrik Olsen}
\author{Paco Villarroya}
\date{}
\title[\tiny Endpoint estimates for compact Calder\'on-Zygmund operators]{Endpoint estimates for compact Calder\'on-Zygmund operators}
\thanks{The second author has been partially supported by Spanish project MTM2011-23164}

\begin{document}

\date{\today}

\begin{abstract}
We prove necessary and sufficient conditions for a Calder\'on-Zygmund operator to be compact at the endpoint from $L^{1}(\mathbb R^{d})$ into $L^{1,\infty}(\mathbb R^{d})$.  
\end{abstract}

\maketitle

\section{Introduction}
The paper \cite{V1} introduced a new $T(1)$ Theory to study compactness of singular integral operators. Its main result provided necessary and sufficient conditions for operators associated with classical Calder\'on-Zygmund kernels to be compact on $L^{p}(\mathbb R)$ for all $1<p<\infty $. 
This characterization was expressed in terms of three conditions: the decay of the derivative of the kernel along the direction of
% parallel to 
the diagonal, %and the action of the operator over special families of functions: 
an appropriate 'weak compactness condition', 
and the membership of properly constructed $T(1)$ and $T^{*}(1)$ functions to the 
space $\CMO(\mathbb R)$. Here, the latter space is defined as the 
closure in $\BMO(\mathbb R)$ of the space of continuous functions vanishing at infinity.
Later, in \cite{PPV}, the endpoint case of compactness from $L^{\infty}(\mathbb R)$ into $\CMO(\mathbb R)$ was obtained. 

We note that, although the results in the two above-mentioned papers were proven in the context of 
functions defined on $\mathbb{R}$, the results and techniques developed also hold in the multi-dimensional setting. See, for instance, the preprint \cite{V2} %developed in parallel with this project and 
which contains the proof of a global $T(b)$ theorem for compactness of singular integrals in $\mathbb R^{d}$. 

A natural question is whether one can obtain the two remaining 
%two type of
 endpoint results, namely, 
compactness from $H^{1}(\mathbb R^{d})$ into $L^{1}(\mathbb R^{d})$
%together with a new direct proof of compactness 
and 
from $L^{1}(\mathbb R^{d})$ into $L^{1,\infty}(\mathbb R^{d})$.
%; and boundedness of the described operators from larger spaces than the classical ones, where they are already known to be compact. 
A little bit of thought shows that the former case is an immediate consequence of \cite{PPV} and Schauder's Theorem, which states that an operator between two Banach spaces,  $T : X \rightarrow Y$, is compact if and only if the same holds true for $T^\ast : Y^\ast \rightarrow X^\ast$ (see e.g. \cite{Runde}). The point of this paper is to prove that the latter endpoint result also holds.

Although 
%We point out that even though these 
both results are the natural extensions of the classical endpoint theorems for boundedness, 
%and they can even be considered as routine results, 
the method used to prove compactness from $L^1(\R^d)$ to $L^{1,\infty}(\R^d)$ is very different from the standard one. It is true that the demonstration follows the same general scheme and shares identical initial steps as 
in the proof of boundedness. 
%and also that its initial steps are identical. 
However, the standard reasoning comes to a halt when applied to 
%collapses as soon as the focus is placed on 
the orthogonal projection operator, 
which is an element completely absent in the classical proof. This difficulty forces one to perform the operator analysis in a different way, more in accordance with the ideas carried out to show compactness at the non-endpoint case \cite{V1}.   

Since the current project is the continuation of  \cite{V1}, we often cite this paper for detailed references about the notation and the definitions we use,
and also for proofs of those results that we merely state. And yet, we intend to 
present a paper as self-contained as possible.

\section{Definitions}

\subsection{Notation}

We say that   $I=\prod_{i=1}^{d}[a_{i},b_{i})$ is a cube in $\mathbb R^{d}$
if the quantity $|b_{i}-a_{i}|$ remains constant for all indices $i\in \{1,\ldots ,d\}$. 
%We call this constant quantity the side length of $I$.
For every cube $I\subset \mathbb R^{d}$ we denote its centre by $c(I)=(2^{-1}(a_{i}+b_{i}))_{i=1}^{d}$, its side length by $\ell(I)=|b_{i}-a_{i}|$, and
its volume by $|I|=\ell(I)^{d}$. For any $\lambda >0$, we denote by $\lambda I$ the cube such that $c(\lambda I)=c(I)$ and $|\lambda I|=\lambda^{d}|I|$. Accordingly, we also write $\mathbb B=\mathbb B^{d}=(-1/2,1/2)^{d}$ and $\mathbb B_{\lambda }=\lambda \mathbb B=(-\lambda/2,\lambda /2)^{d}$.

We denote by $|\cdot |_{p}$, with $0<p\leq \infty $, the $\ell^{p}$-norm in $\mathbb R^{d}$ and by $|\cdot |$ the modulus of a complex number. Hopefully, this notation will not cause any confusion with the one we use for the 
volume of a cube. 
%\begin{notation}\label{ecandrdist}

Given two cubes $I,J\subset \mathbb R^{d}$,
we denote by $\langle I,J\rangle$ any cube with minimal side length containing $I\cup J$ and
write its side length by $\diam(I\cup J)$. If there is more than one cube satisfying these conditions, we will simply select one and refer to it as $\langle I,J\rangle$ regardless of the choice.

We note that if $I=\prod_{i=1}^{d}I_{i}$, $J=\prod_{i=1}^{d}J_{i}$, with 
$I_{i},J_{i}$ intervals in $\mathbb R$,  
we have $\diam(I\cup J)=\max_{i}\diam(I_{i}\cup J_{i})$, where $\diam(I_{i}\cup J_{i})$
is the length of $\langle I_{i},J_{i}\rangle $, the smallest interval containing $I_{i}$ and $J_{i}$. 
Therefore, we have the following equivalences
\begin{align*}
\diam(I\cup J) & \approx  \frac{\ell(I) + \ell(J)}{2}+|c(I)-c(J)|_{\infty } \\
& \approx \max (\ell(I),\ell(J) )+|c(I)-c(J)|_{\infty}.
\end{align*}

We also define
the relative distance between $I$ and $J$ by
$$\rdist(I,J)=\frac{\diam(I\cup J)}{\max(\ell(I),\ell(J))},$$ 
which is comparable to $\max(1,n)$ where $n$ is  
the smallest number of times the larger cube needs to be shifted a distance equal to its side length 
%along the line passing through the centers of both cubes 
so that it contains the smaller one. 
%which is comparable to the smallest number of times the larger cube needs to be shifted 
%%along the line passing through the centers of both cubes 
%so that it contains the smaller one. 
Note that from the above, we have
\begin{align*}
\frac{1}{2}\Big(1+\frac{|c(I)-c(J)|_{\infty }}{\max(\ell(I),\ell(J))}\Big)\leq \rdist(I,J) \leq 1+\frac{|c(I)-c(J)|_{\infty }}{\max(\ell(I),\ell(J))}.
%& \approx  1+\frac{\dist_{\infty }(I,J)}{\max(\ell(I),\ell(J))}.
\end{align*}

We also define the eccentricity of $I$ and $J$ to be
$$
\ec(I,J)=\frac{\min(|I|,|J|)}{\max (|I|,|J|)}.
$$

%\end{notation}

Finally, 
we say that a cube $I$ is dyadic if $I=2^{j}\prod_{i=1}^{d}[k_{i},k_{i}+1)$ for some $j,k_{1}, \ldots ,k_{d}\in \mathbb Z$, and denote by ${\mathcal C}$ and ${\mathcal D}$ the families of all cubes and all dyadic cubes in $\mathbb R^{d}$, respectively.

\begin{definition} \label{Imdef}
For every $M\in \mathbb N$, we define ${\cal C}_{M}$ to be the family of all cubes in $\mathbb R^{d}$ such that
$2^{-M}\leq \ell(I)\leq 2^{M}$ and
$\rdist(I,\mathbb B_{2^{M}})\leq M$.  We also define ${\cal D}_{M}$ to be the intersection of ${\cal C}_{M}$ with ${\mathcal D}$. 

For every fixed $M$, we will call the cubes in ${\cal C}_{M}$ and ${\cal D}_{M}$   lagom\footnote{ `Lagom' is a  Swedish word with the following meanings: adequate, moderate, in balance, just right.} cubes and dyadic lagom cubes respectively.
\end{definition}

\begin{remark} \label{remark in section 2}
Note that $I\in {\cal C}_{M}$ implies that  $2^{-M}(2^{M}+|c(I)|_{\infty })\leq M$, and so
$|c(I)|_{\infty }\leq (M-1)2^{M}$.
Therefore, in this case, $I\subset \mathbb B_{M2^{M}}$ with
$2^{-M}\leq \ell(I)$.

On the other hand, $I\notin {\cal C}_{M}$ implies either $\ell(I)>2^{M}$ or
$\ell(I)<2^{-M}$, or $2^{-M}\leq \ell(I)\leq 2^{M}$ with $|c(I)|_{\infty }>(M-1)2^{M}$.
%Note that it is possible that $I\notin {\cal D}_{M}$ and $I\subset \mathbb B_{(M+1)2^{M}}$ arbitrarily close to %the origin.
\end{remark}

\subsection{Compact Calder\'on-Zygmund kernels and associated operators}

We define the type of kernels that can be associated with compact operators. 
\begin{definition}\label{admissible}
Three bounded functions $L,S, D: [0,\infty )\rightarrow [0,\infty )$
constitute a set of admissible functions if 
%they 
%satisfy that the product $L\cdot S$ and $D$ 
%are bounded and 
the following limits hold
%there is a $\theta >0$ such that the following conditions hold:
\begin{equation}\label{limits}
\lim_{x  \rightarrow \infty }L(x )
=\lim_{x \rightarrow 0}S(x )=\lim_{x  \rightarrow \infty }D(x )=0.
%\item $F(2x)\leq C2^{\theta }F(x)$ and $F(2^{-1}x)\leq C2^{\theta }F(x)$ for all $x>0$
\end{equation}
%and that $L,D$ are monotone non-increasing while $S$ is monotone non-decreasing. 

\end{definition}

\begin{remark}\label{constants}
Since any fixed dilation of an admissible function  is again admissible, we will often omit all universal constants appearing in the argument of these functions.
% Moreover, for this reason, we note that the use of any other norm $|\cdot |_{p}$ in (\ref{limits}) would provide an equivalent definition. 
\end{remark}

\begin{definition}\label{prodCZ}
%Let $\Delta $ be the diagonal of $\mathbb R^{d}$.
A function $K:(\mathbb R^{d}\times \mathbb R^{d}) \setminus \{ (t,x)\in \mathbb R^{d}\times \mathbb R^{d}: t=x\}\to \mathbb C$ is called a
compact Calder\'on-Zygmund kernel if it is bounded in its domain and there exist
$0<\delta < 1$, $C>0$, and admissible functions $L,S,D$ such that
\begin{equation}\label{CCZ}
|K(t,x)-K(t',x')|
\le 
%\frac{1}{(|x|+|t|)^{\theta}} 
C\frac{(|t-t'|_{\infty }+|x-x'|_{\infty })^\delta}{|t-x|_{\infty }^{d+\delta}}F(t,x),
%L(|x-t|)S(|x-t|)D(|x+t|)
\end{equation}
whenever $2(|t-t'|_{\infty }+|x-x'|_{\infty })<|t-x|_{\infty }$, where 
$$
F(t,x)=L(|t-x|_{\infty })S(|t-x|_{\infty }) D(|t+x|_{\infty }).
$$ 

%and $|x|+|t|>M$.
\end{definition}

We use the standard definition of multi-indices: $\alpha=(\alpha_{1},\ldots , \alpha_{d})\in \mathbb N^{d}$,
$|\alpha |=\sum_{i=1}^{d}\alpha_{i}$ and
%$\partial^{\alpha}=\prod_{i=1}^{n}\partial_{i}^{\alpha_{i}}$ and
$\partial^{\alpha}=\frac{\partial^{|\alpha |}}{\partial^{\alpha_{1}}_{x_{1}}\cdots \, \partial^{\alpha_{d}}_{x_{d}}}$.

\begin{definition}\label{SchwartzN} 
%In a similar way the space of Schwartz functions $\S(\R)$ is constructed,
For every $N\in \mathbb N$, $N\geq 1$, we define $\S_{N}(\R^{d})$ to be the set of 
all functions 
%$f$ differentiable up to order $N$, 
$f\in {\mathcal C}^{N}(\mathbb R^{d})$
such that 
$$
\| f\|_{m,n}=\sup_{x\in \mathbb R^{d}}|x|^{\beta }|\partial^{\alpha}f(x)|<\infty
$$ 
for all $\alpha ,\beta \in \mathbb N^{d}$ with $|\alpha |,|\beta |\leq N$. Clearly, $\S_{N}(\R^{d})$ 
equipped with the family of seminorms $\| \cdot \|_{\alpha ,\beta }$
is a Fr\'echet space. Then,   
we can also define its dual space $\S_{N}(\R^{d})'$ equipped with the dual topology which turns out to be
a subspace of the space of multidimensional tempered distributions. We write $\S(\R^{d})$ for the classical Schwartz space.
\end{definition}

\begin{definition}\label{intrep}
Let $T:\S_{N}(\R^{d})\to \S_{N}(\R^{d})'$ be a linear operator which is continuous with respect to the topology 
of $\S_{N}(\R^{d})$ and the dual topology of  $\S_{N}(\R^{d})'$.

We say that 
$T$ is associated with a compact Calder\'on-Zygmund kernel $K$ if for all $f,g\in {\cal S_{N}(\R^{d})}$
with disjoint compact supports, 
the action of $Tf$ as a distribution satisfies the following 
integral representation
$$
\langle Tf, g\rangle =\int_{\R^{d}}\int_{\R^{d}} f(t)g(x) K(t, x)\, dt \, dx.
$$
\end{definition}

\subsection{The weak compactness condition}

\begin{definition}\label{adapted}
For $0<p\leq \infty $, we say that a function $\phi \in \S_{N}(\R^{d})$ is
an $L^p(\mathbb R^{d})$-normalized bump function adapted to $I$ with constant $C>0$ and order 
%$N'\in \mathbb N$ and decay 
$N\in \mathbb N$ if, for all multi-indices $0\leq \abs{\alpha} \leq N$, it holds that
$$
|\partial^{\alpha}\phi(x)|\le \frac{C}{|I|^{1/p}\ell(I)^{|\alpha|}}  \bigg(1+ \frac{|x-c(I)|_{\infty }}{\ell(I)}\bigg)^{-N}.
%\, \, \, 0\leq  |\alpha |\le N,
$$
\end{definition}

Observe that, for $Np>d$, the bump functions in Definition \ref{adapted} are normalized to be uniformly bounded in $L^p(\mathbb R^{d})$. The order of the bump functions
will always be denoted by $N$, even though its value might change from line to line.
We will often use the greek letters $\phi $, $\varphi $ for general bump functions while
we reserve the use of $\psi $ to denote bump functions with mean zero.
If not otherwise stated, we will usually assume that the bump functions are $L^2(\mathbb R^{d})$-normalized.

We now state the weak compactness condition. 
\begin{definition}\label{WB}
A linear operator $T : \S_{N}(\R^{d}) \to \S_{N}(\R^{d})'$ satisfies the weak compactness condition if 
there exist admissible functions $L, S, D$ 
%and a constant $C_T > 0$ 
such that: for every $\epsilon>0$ there 
exists $M\in \mathbb N$ so that for any cube $I$ and every pair $\phi_I, \varphi_I$ of
$L^2$-normalized bump functions adapted to $I$ with constant $C>0$ and order $N$, we have
%\begin{equation}\label{weakcompactness}
$$
|\langle T\phi_I,\varphi_I)\rangle |\lesssim C \bigg(L\Big(\frac{\ell(I)}{2^M}\Big) \cdot S\Big(2^{M}\ell(I)\Big) \cdot
D\Big(\frac{\rdist\big(I,\mathbb B_{2^{M}})}{M}\Big)+\epsilon \bigg),
$$
where the implicit constant only depends on the operator $T$.

%\end{equation}
\end{definition}

There are other alternative and less technical formulations of this concept. For example, we can say 
%, which follows immediately from the one above, is as follows: 
that 
$T$ satisfies the weak compactness condition if and only if, for every pair $\phi_I, \varphi_I$ of
$L^2$-normalized bump functions adapted to $I$, we have
\begin{align*}
%(i) &\quad 
\lim_{M\rightarrow \infty }\sup_{I\notin {\mathcal D}_{M}}|\langle T\phi_I,\varphi_I)\rangle |=0, 
%\\
%(ii) &\quad 
%\lim_{M\rightarrow \infty }|\langle (P_{M}^{\perp}\circ T)(\phi_I),\varphi_I)\rangle |=0, \quad \forall I \in 
%\mathcal{D}.
\end{align*}
where the lagom dyadic cubes  ${\mathcal D}_{M}$  
appear in Definition \ref{Imdef}. 
However, we prefer the formulation used in Definition \ref{WB} because
it is particularly well-suited for the calculations performed in \cite{V1} and thus, the ones carried out in the current paper. 
% to develop our theory of compact Calder\'on-Zygmund operators. 

We introduce the following notation to simplify otherwise cumbersome formulas, which appear both in the statement of Proposition \ref{symmetricspecialcancellation} and in the proof of Theorem \ref{theoweakL1}, below.  Namely, we write
\begin{align*}
F(I&;M)=L_{K}\big(\ell(I)\big)\cdot S_{K}\big(\ell(I)\big)\cdot D_{K}\big(\rdist(I,\mathbb B)\big) 
\\
&+F_{W}\Big(\frac{\ell(I)}{2^M}\Big) \cdot S_{W}\Big(2^{M}\ell(I)\Big) \cdot D_{W}\Big(\frac{\rdist(I,\mathbb B_{2^{M}})}{M}\Big),
\end{align*}
where $L_{K}$, $S_{K}$, $D_{K}$ are the functions appearing in the definition of a compact 
Calder\'on-Zygmund kernel,
while $L_{W}$, $S_{W}$, $D_{W}$ and the constant $M$ are as in the definition of the weak compactness condition.
% Note that  $M $ depends both on $T$ and  $\epsilon$. 
%that the dependence of $\epsilon $ in $F_{T}$ is due to the fact that 
We  also set 
\begin{align*}
F(I_{1},\ldots, &I_{n};M) =  \sum_{i=1}^{n}L_{K}\big(\ell(I_{i})\big) \cdot \sum_{i=1}^{n}S_{K}\big(\ell(I_{i})\big) \cdot
 \sum_{i=1}^{n}D_{K}\big(\rdist(I_{i},\mathbb B)\big)\\
&+ \sum_{i=1}^{n}L_{W}\Big(\frac{\ell(I_{i})}{2^M}\Big) \cdot \sum_{i=1}^{n}S_{W}\Big(2^{M}\ell(I_{i})\Big)
\cdot \sum_{i=1}^{n}D_{W}\Big(\frac{\rdist(I_{i},\mathbb B)}{M}\Big).
\end{align*}

The following lemma is proven at the beginning of the proof of Theorem \ref{Mainresult}, below, as it is given in \cite{V1}.
\begin{lemma} \label{lesizzy}
Given 
 $\epsilon >0$, then there exists  exists $M_0$
%depending on $\epsilon, M_{K}, M_{T}$ 
so that for all  $M>M_{0}$ 
we have $F(I_{1},\ldots ,I_{6};M_{T,\epsilon })\lesssim \epsilon $ whenever all $I_{i}\in {\cal D}_{M}^{c}$.
%This will follow as a consequence of the inequality $F(I;M_{T,\epsilon })\lesssim \epsilon $  
%We note that the implicit constant only depends on the admissible functions. 
\end{lemma}

We end this subsection with two results that we will use to prove the reverse implication in our main result. Their proofs can be found in \cite[Theorem 10.1]{NTV} and \cite[Theorem 3.1]{CobPer}, respectively.

\begin{theorem}\label{fromweak2strong}
Let $T$ be an operator with a standard Calder\'on-Zygmund kernel and bounded from $L^{1}(\mathbb R^{d})$ into $L^{1,\infty }(\mathbb R^{d})$. Then, $T$ is bounded on $L^{p}(\mathbb R^{d})$ for any $1<p<\infty$ with
$
\| T\|_{L^{p}(\mathbb R^{d})\rightarrow L^{p}(\mathbb R^{d})}\lesssim \| T\|_{L^{1}(\mathbb R^{d})\rightarrow L^{1,\infty }(\mathbb R^{d})}
$
and the implicit constant only depends on $p$ and the dimension $d$.
\end{theorem}

\begin{theorem}\label{compactinterpolation}
Let $A=(A_{0},A_{1})$ and $B=(B_{0},B_{1})$ be quasi-Banach couples and let $T:A\rightarrow B$ such that 
 $T:A_{0}\rightarrow B_{0}$ compactly. Then, for any $0<\theta <1$ and $0<q\leq \infty $, 
  $T:(A_{0},A_{1})_{\theta ,q}\rightarrow (B_{0},B_{1})_{\theta ,q}$ is compact.  
\end{theorem}

\subsection{Characterization of compactness and the lagom projection operator }

The following 
characterization of compact operators in a Banach space with a Schauder basis (see for example \cite{Fab}) was used in 
\cite{V1} to study compact Calderon-Zygmund operators. 

\begin{theorem}\label{charofcompact}
Suppose that $\{e_{n}\}_{n\in \mathbb N}$ is a Schauder basis of a Banach space $E$. For each positive integer $k$, let $P_{k}$ be the canonical projection
$$
P_{k}\Big(\sum_{n\in \mathbb N}\alpha_{n}e_{n}\Big)=\sum_{n\leq k}\alpha_{n}e_{n}.
$$
Then, a bounded linear operator $T:E \to E$ is compact if and only if $P_{k}\circ T$ converges to $T$ in operator norm.
\end{theorem}

Let $E$ be one of the following Banach spaces: the Lebesgue space $L^{p}(\mathbb R^{d})$, $1<p<\infty $, the Hardy space $H^{1}(\mathbb R^{d})$, or the space $\CMO(\R^{d})$, defined in Subsection \ref{defCMO} below. 
%follows.
%to be introduced later as the closure in $\BMO(\mathbb R^{d})$ of continuous functions vanishing at infinity. 
In each of these cases, $E$ is equipped with a wavelet basis which is also a Schauder basis (see \cite{HerWeiss} and Lemma \ref{lemacmochar}). Moreover, in these cases, we can assume 
that the wavelets belong to $\S_{N}(\R^{d})$ and, if needed, that they are compactly supported. 
%the compact support of the wavelets. 
%In fact, we make this assumption throughout the remainder of the paper.
However, we intentionally decide to use more general wavelets to explicitly show that our results hold in settings where, for example, compactly supported wavelets are not available.

\begin{definition}\label{lagom}
Let $E$ be one of the previously mentioned Banach spaces.  
Let $(\psi_{I}^i)_{I\in {\mathcal D}, i=1,\ldots ,2^{d}-1}$ be a normalized wavelet basis of $E$ and $(\tilde{\psi}^{i}_I)_{\in \mathcal{D}, i=1,\ldots ,2^d-1}$ its dual wavelet basis. 
Then, for every $M\in \mathbb N$,
we define the lagom projection operator 
$$
P_{M}f=\sum_{I\in {\cal D}_{M}}\sum_{i=1}^{2^{d}-1}\langle f,\tilde{\psi}_{I}^{i}\rangle \psi_{I}^{i},
$$
where $\langle f,\tilde\psi_{I}^{i}\rangle =\int_{\mathbb R^{d}}f(x)\overline{\tilde\psi_{I}^{i}(x)}dx$.
\end{definition}

We also define $P_{M}^{\perp }f=f-P_{M}f$, and we  remark that the equality
\begin{equation}\label{ortho}
P_{M}^{\perp}f=\sum_{I\in {\cal D}_{M}^{c}}\sum_{i=1}^{2^{d}-1}\langle f, \tilde\psi_{I}^{i}\rangle \psi_{I}^{i}
\end{equation}
is to be interpreted in the sense of  Schauder bases, i.e.,
$$
\lim_{M'\rightarrow \infty }\Big\| P_{M}^{\perp}f
-\sum_{I\in {\cal D}_{M'}\backslash {\cal D}_{M}}\sum_{i=1}^{2^{d}-1}\langle f,\tilde\psi_{I}^{i}\rangle \psi_{I}^{i}  \Big\|_{E}=0.
$$

In the language of the lagom projection, we can give yet another alternative formulation of weak compactness  (Definition \ref{WB}). Namely, an operator $T$ is weakly compact if and only if, for every pair $\phi_I, \varphi_I$ of
$L^2$-normalized bump functions adapted to $I$, we have
\begin{align*}
%(i) &\quad \lim_{M\rightarrow \infty }\sup_{I\notin {\mathcal D}_{M}}|\langle T\phi_I,\varphi_I)\rangle |=0, 
%\\
%(ii) &\quad 
\lim_{M\rightarrow \infty }|\langle (P_{M}^{\perp}\circ T)(\phi_I),\varphi_I)\rangle |=0, \quad \forall I \in 
\mathcal{D}.
\end{align*}
%Here, the definition of ${\mathcal D}_{M}$  
%appears in Definition \ref{Imdef}, and the   lagom projection $P_{M}^{\perp }$ is introduced  below in Definition \ref{lagom}. 

Strictly speaking, the characterization given in Theorem \ref{charofcompact}   is not sufficient for our purposes since, in Section \ref{L1weak}, we also consider compact operators into the space $L^{1,\infty}(\R^d)$, which is a quasi-Banach space.  This is addressed in Definition \ref{compactnessonweakL1}, 
%The following definition 
where we define compact operators from $L^1(\R^d)$ into $L^{1,\infty}(\R^d)$ in the topological sense.
\begin{definition}\label{compactnessonweakL1}
An operator $T: L^{1}(\mathbb R^{d})\rightarrow L^{1,\infty }(\mathbb R^{d})$ is compact if, for every bounded set 
$A\subset  L^{1}(\mathbb R^{d})$, the set $T(A)$ is relatively compact in $L^{1,\infty }(\mathbb R^{d})$. Equivalently, $T$ is compact if, 
for every sequence $(f_{n})_{n\in \mathbb N}\subset  L^{1}(\mathbb R^{d})$ with $\| f_{n}\|_{L^{1}(\mathbb R^{d})}\lesssim 1$, %the sequence 
%$(T(f_{n}))_{n\in \mathbb N}\subset  L^{1,\infty }(\mathbb R^{d})$ contains a convergence subsequence, i.e., 
there exist 
a subsequence $(f_{n_{k}})_{k\in \mathbb N}$ and 
$g\in L^{1,\infty }(\mathbb R^{d})$ such that 
$
\lambda m(\{ x \in \mathbb R^{d} : |Tf_{n_{k}}(x)-g(x)|>\lambda \} )
$
tends to zero when $k$ tends to infinity uniformly for all $\lambda >0$.
\end{definition}

\begin{remark}\label{resizzy}
Observe  that finite rank operators are compact in this sense, and that the limit of finite rank operators is a compact operator. 
\end{remark}
We also note that, in light of Theorem \ref{charofcompact}, it would be natural to assume that the above definition is equivalent to asking that $P_M^\perp T$ converges to zero in the operator norm $\| \cdot \|_{L^{1}(\mathbb R^{d})\rightarrow L^{1,\infty }(\mathbb R^{d})}$. However, this is not the case as we see from the following example.
\begin{example}\label{counterexample}
        Let $(\psi_{I})_{I\in \mathcal D}$ be the Haar wavelet of $L^{2}(\R)$ and $P_{M}$ the associated lagom projection operator. 
	Then, the operator defined by $$Tf  = \langle f,\psi_{[0,1]}\rangle \chi_{[0,1]}$$ is compact from $L^1(\R)$ to $L^{1,\infty}(\R)$ (since it is bounded and of finite rank), but $P_M^\perp T$ does not converge to zero in $L^{1,\infty}(\R)$. Indeed, it follows from the computation below   that $P_M^\perp T \psi_{[0,1]}= 2^{-M} \chi_{[0,2^M]}$, whence  $\norm{P_M^\perp T}_{L^1(\R) \rightarrow L^{1,\infty}(\R)}  \geq 1$ for all $M\in \N$.
%	\begin{equation*}
%		\norm{P_M^\perp T}_{L^1 \rightarrow L^{1,\infty}} \geq 1.
%	\end{equation*}

First, we observe that
$$
P_{M}^{\perp }T\psi_{[0,1]}=P_{M}^{\perp}\chi_{[0,1]}
=\chi_{[0,1]}-\sum_{I\in \mathcal D_{M}}\langle \chi_{[0,1]},\psi_{I}\rangle \psi_{I}.
%=\sum_{K\notin \mathcal D_{M}}\langle \varphi_{I},\psi_{K}\rangle \psi_{K}(x)
$$
%Since $\langle \tilde{\chi}_{I},\psi_{J}\rangle=0$ if $J\cap I\neq \emptyset $ or $J\subseteq I$, 
Now, $\langle \chi_{[0,1]},\psi_{I}\rangle\neq 0$ if and only if $I=(0,2^{k})$ with $1\leq k\leq M$ and, in that case, 
$\langle \chi_{[0,1]},\psi_{I}\rangle=|I|^{-1/2}$.
With this, we obtain 
\begin{align*}
P_{M}^{\perp }T\psi_{[0,1]}
&=\chi_{(0,1)}-\sum_{1\leq k\leq M}2^{-\frac{k}{2}}2^{-\frac{k}{2}}(\chi_{(0,2^{k-1})}-\chi_{(2^{k-1},2^{k})})
\\
&=2^{-M}\chi_{(0,1)}+
\sum_{1\leq j\leq M}2^{-M}\chi_{(2^{j-1},2^{j})}
= 2^{-M}\chi_{(0,2^{M})}
%>\lambda_{M}\chi_{(0,2^{M})}(x)
\end{align*}
%Therefore, 
%$$
%\lambda_{M} m(\{ x\in \mathbb R: |P_{M}^{\perp }T_{b}^{*}(f)(x)|>\lambda_{M}\})
%\geq 2^{-(M+1)}2^{M}=2^{-1}
%$$
as claimed.
\end{example}

\subsection{The space $CMO(\mathbb R^{d})$ and the construction of $T(1)$}\label{defCMO}
\begin{definition}
We define $\CMO(\mathbb R^{d})$ as the closure in $\BMO(\mathbb R^{d})$ of the space of continuous functions vanishing at infinity. 
\end{definition}

The next lemma gives two characterizations of $\CMO(\mathbb R^{d})$: the first one in terms of the average deviation from the mean, and the second one in terms of a  
wavelet decomposition. 
See \cite{Roch} and \cite{LTW} for  the proofs.
We will only use the latter formulation.

\begin{lemma} \label{lemacmochar}The following statements are equivalent:
\begin{enumerate}
\item[(i)] $f\in \CMO(\mathbb R^{d}),$

\item[(ii)] $f\in \BMO(\mathbb R^{d})$ and
\begin{equation*}\label{CMO}
\lim_{M\rightarrow \infty } \sup_{I\notin {\mathcal I}_{M}} 
\frac{1}{|I|}\int_{I}\Big|f(x)-\frac{1}{|I|}\int_{I}f(y)dy\Big|dx =0,\\
\end{equation*}
%where $m_{I}(f)=\frac{1}{|I|}\int_{I}f(x)dx$.

\item[(iii)] $f\in \BMO(\mathbb R^{d})$ and
\begin{equation*}\label{CMO2}
\lim_{M\rightarrow \infty} 
%\| P_{M}^{\perp}(f)\|_{\BMO(\mathbb R)}
\sup_{\Omega \subset \mathbb R^{d}} \Big(\frac{1}{|\Omega |}
\sum_{\tiny \begin{array}{c}I\notin {\cal D}_{M}\\ I\subset \Omega \end{array}}
\sum_{i=1}^{2^{d}-1}|\langle f,\psi_{I}^{i}\rangle |^{2}\Big)^{1/2}
=0,
\end{equation*}
where the supremum is taken over all measurable sets $\Omega \subset \mathbb R^{d}$.
\end{enumerate}
%As a consequence, $(\psi_I)_{I\in \mathcal{D},i=1,\ldots, 2^{d}-1}$ is a Schauder basis for $\CMO(\R^{d})$.
\end{lemma}

Next, we state a
technical lemma needed  to give meaning to $T(1)$ and $T^{*}(1)$.
To this end, we introduce some notation.
For $a \in \mathbb{R}$ and $\lambda>0,$ we define the translation operator
as $T_af(x) = f(x-a)$ and the dilation operator as $D_\lambda f(x) = f(x/\lambda)$. 
Let $\Phi \in \mathcal{S}(\mathbb R^{d})$ be such that $\Phi(x) = 1$ for $|x|_{\infty } \leq 1$, $0 < \Phi(x)<1$ for
$1 < |x|_{\infty } < 2$ and $\Phi(x) = 0$ for $|x|_{\infty }>2$.
\begin{lemma}\label{definecmo}
Let $T$ be a linear operator associated with a compact Calder\'on-Zygmund kernel $K$ with parameter 
$0<\delta <1$.
% (as in Section \ref{kernelandadmissible}).

Let $I\subset \mathbb R^{d}$ be a cube and let $f\in \S_{N}(\R^{d})$ have compact support in $I$ and mean zero.
%$H^1$-atom on $S$
Then, the limit
$${\mathcal L}(f)=\lim_{k\to \infty} \langle T(
{\cal T}_{a} 
{\mathcal D}_{2^{k}\ell(I)}\Phi ),f\rangle $$
exists 
and is independent of 
the translation parameter $a\in \mathbb R$ and 
the cut-off function $\Phi$.

\end{lemma}

The previous lemma allows one to define $T(1)$ as an element on the dual of the space of functions in 
$\S_{N}(\R^{d})$ with compact support and mean zero. Namely, define $\langle T(1), f\rangle = {\mathcal L}(f)$ for all $f \in \S_{N}(\R^{d})$.

\subsection{Compactness on $L^{p}(\mathbb R^{d})$}
We now state the main result in \cite{PPV}, whose proof, although proven only for the one-dimensional case, also holds in the setting of several variables. 
\begin{theorem}\label{Mainresult}
Let
%$K$ be a Calderon-Zygmund kernel with parameter $\delta$ and
$T$ be a linear operator associated with a standard Calder\'on-Zygmund kernel. 

Then, $T$ extends to a compact operator on $L^p(\mathbb R^{d})$, for any $1 < p< \infty$, if and only if $T$ 
is associated with a compact Calder\'on-Zygmund kernel and it satisfies both
the weak compactness condition
and the cancellation conditions
$T(1), T^{*}(1) \in \CMO(\mathbb R^{d})$. 

Under the same hypotheses, $T$ is also compact as a map from $L^{\infty}
%_\CMO
(\mathbb R^{d})$ into 
$\CMO(\mathbb R^{d})$. Moreover, with the extra assumption $T(1)=T^{*}(1)=0$, $T$ is compact 
from $\BMO(\mathbb R^{d})$ into $\CMO(\mathbb R^{d})$.

\end{theorem}

We end this section stating 
the main auxiliary result in the proof of Theorem \ref{Mainresult}, which is also the starting point of the proof of the endpoint result in this paper. To this end, we provide the following definitions:
given two cubes $I$ and $J$, we denote $K_{min}=J$ and $K_{max}=I$ if $\ell(J)\leq \ell(I)$, and $K_{min}=I$ and $K_{max}=J$ otherwise. We denote by $\tilde{K}_{max}$ the translate of $K_{max}$ with the same centre as $K_{min}$.

\begin{proposition}\label{symmetricspecialcancellation}
Let $T$ be a linear operator associated with a compact Calder\'on-Zygmund kernel with parameter $\delta$. 
We assume $T$ satisfies the weak compactness condition
%(\ref{linearrestricted}) 
and the special cancellation condition $T(1)=0$ and $T^{*}(1)=0$.

Then, 
%Then, for any $0<\delta'<\delta$ and for every  
for any $\theta\in (0,1)$ small enough, 
there exist $0<\delta'<\delta $, $N\geq 1$ 
%large enough 
and $C_{\delta'}>0$ such that
for every $\epsilon >0$, all cubes $I, J$ 
and 
%$\psi\in {\cal S}(\mathbb R^2)$ be adapted to $I_1\times I_2$. Assume that 
all mean zero bump functions $\psi_{I}$, $\psi_{J}$, $L^{2}$-adapted to $I$ and $J$ respectively
with constant $C>0$ and order $N$, 
we have
%there exists $>0$  such that
\begin{equation*}\label{twobump2}
|\langle T\psi_{I},\psi_{J}\rangle |\leq  C_{\delta'} C \, \frac{\ec(I,J)^{\frac{1}{2}+\frac{\delta'}{d}}}{ \rdist(I,J)^{d+\delta'} }
%\left( \max(|I|,|J|)^{-1}\diam(I\cup J) \right)^{-(1+\delta')}
\Big(F(I_{1},\ldots ,I_{6};M_{T,\epsilon })+\epsilon \Big),
%\Big(\max_{i=1,2,3}F(I_{i})+\epsilon \Big)
%L(|I_{i}|)S(|I_{i}|^{-1})D(\rdist(I_{i},\mathbb B_{2^{M}}))
\end{equation*}
where
%$\delta'=(1-\epsilon')\delta $, 
$I_{1}=I$, $I_{2}=J$, $I_{3}=\langle I, J\rangle $, $I_{4}=\lambda_{1}\tilde{K}_{max}$, 
$I_{5}=\lambda_{2}\tilde{K}_{max}$ and
$I_{6}=\lambda_{2}K_{min}$
with $\lambda_{1}=\ell(K_{max})^{-1}{\rm diam}(I\cup J)$,  
$\lambda_{2}=\ell(K_{min})^{-\theta}{\rm diam}(I\cup J)^{\theta }$. 
\end{proposition}

\section{Localization properties of bump functions}

In this section, 
we prove two technical results. Lemma \ref{bumpsinteracting}  concerns  the localization of multi-variable bump functions 
while  Lemma \ref{lowoscillation}   estimates 
 the interaction of bump functions with atoms. 
The proofs of both results 
in the one-dimensional case can be found in \cite{TLec}. See also \cite{V1} for a more detailed proof of the latter result.

\begin{lemma}\label{bumpsinteracting}
Let $\phi_{I}$ and $\psi_{J}$ be bump functions $L^2$-adapted to $I$ and $J$ respectively with order 
%zero, decay 
$N\geq d$ and constant $C$. For $\ell(J)\leq \ell(I)$, then
\begin{equation}\label{1/2}
|\langle \phi_{I}, \psi_{J}\rangle |\lesssim  C^{2}\bigg(\frac{|J|}{|I|}\bigg)^{\frac{1}{2}}  \bigg(1+ \frac{|c(I)-c(J)|_{\infty }}{\ell(I)}\bigg)^{-N}.
\end{equation}
If, in addition, $\phi_{I}$ and $ \psi_J$  have order $N>d$ and $\psi_{J}$ has vanishing mean, i.e., $\int \psi_{J}(x)dx=0$, then
\begin{equation}\label{1/2+1/d}
|\langle \phi_{I}, \psi_{J}\rangle|\lesssim  C^{2}\bigg(\frac{|J|}{|I|}\bigg)^{\frac{1}{2} + \frac{1}{d}}
\bigg(1+ \frac{|c(I)-c(J)|_{\infty }}{\ell(I)}\bigg)^{-N+d}.
%(1+ \ell(I)^{-1}|c(I)-c(J)|_{\infty })^{-N+d}.
%\Big(1 + \frac{|c(I)-c(I')|}{|I'|}\Big)^{-N+1}
\end{equation}
\end{lemma}

\begin{proof}
We start by proving inequality (\ref{1/2}).
Let $c$ be the midpoint between $c(I)$ and $c(J)$, let $L$ be the line going through $c(I)$ and $c(J)$ and let $H\subset \mathbb R^{d}$ be the hyperplane perpendicular to $L$  passing through $c$.
Let also $H_{I}$ and $H_{J}$ be the two half-spaces defined by the connected components of $\mathbb R^{d}\backslash H$ so that $c(I) \in H_{I}$ and  $c(J) \in H_{J}$. We split
$$
\langle \phi_{I},\psi_{J}\rangle = \int_{H_{I}}\phi_{I}(x)\overline{\psi_{J}(x)}dx
+\int_{H_{J}}\phi_{I}(x)\overline{\psi_{J}(x)}dx.
$$
Applying H\"older's inequality and Definition \ref{adapted}, we get 
\begin{multline} \label{the start}
|\langle \phi_{I},\psi_{J}\rangle | \leq \| \phi_{I}\|_{L^{1}(\mathbb R^{d})}\| \psi_{J}\|_{L^{\infty }(H_{I})}
+\| \phi_{I}\|_{L^{\infty }(H_{J})} \| \psi_{J}\|_{L^{1}(\mathbb R^{d})}
\\
\leq  C^{2}\bigg(\frac{|I|}{|J|}\bigg)^{\frac{1}{2}} \bigg(1 + \frac{|c-c(J)|_{\infty }}{\ell(J)}\bigg)^{-N}
+C^{2}\bigg(\frac{|J|}{|I|}\bigg)^{\frac{1}{2}}  \bigg(1 + \frac{|c-c(I)|_{\infty }}{\ell(I)}\bigg)^{-N}.
\end{multline}
Since $|c-c(I)|_{\infty }=|c-c(J)|_{\infty }=|c(I)-c(J)|_{\infty }/2$, $\ell(J)\leq \ell(I)$ and  $N\geq d$, we have that the first term is smaller than the second one, which is of the desired form. 

\vskip10pt
To prove (\ref{1/2+1/d}), 
%we begin by introducing some notations. Take $c$ to be the midpoint of $c(I)$ and $c(J)$. For the sake of simplicity, 
we assume without loss of generality that 
$|c(I)-c(J)|_{\infty }=|c(I_{1})-c(J_{1})|$. Then, for all $x\in \mathbb R^{d}$ we   write $x=(x_{1},x')$ with $x'\in \mathbb R^{d-1}$.
We define the operators
%$\{ \lambda x: \lambda <1\}$. 
$$
D^{-1}_{1}(\psi_{J})(x)=\int_{-\infty }^{x_1}\psi_{J}(s,x')ds
$$
and, for $t \in \R$, 
$$
D^{-1}(\psi_{J})(t)=\int_{\mathbb R^{d-1}}\int_{-\infty }^{t}\psi_{J}(x_1,x')dx_1dx'
=\int_{te_{1}+H^{-}_{1}}\psi_{J}(x)dx,
$$
where $H^{-}_{1}=\{ x\in \mathbb R^{d}: x_{1}\leq 0\}$. 
Note that, due to the vanishing mean  of 
$\psi_{J}$, we have 
\begin{equation} \label{the switch}
D^{-1}(\psi_{J})(t)=-\int_{\mathbb R^{d-1}}\int_{t}^{\infty }\psi_{J}(x_1,x')dsdx'
=-\int_{te_{1}+H_{1}^{+}}\psi_{J}(x)dx,
\end{equation}
where $H_{1}^{+}=\{ x\in \mathbb R^{d}: 0\leq x_{1}\} $. 
Then, it is readily   checked that 
%we make the computation
\begin{align*}
\langle \phi_{I}, \psi_{J}\rangle &=
- \int_{\mathbb R^{d}}\partial_{1}\phi_{I}(x) \cdot 
D^{-1}_{1}\overline{\psi_{J}}(x)dx.
\end{align*}
Now, the function $D^{-1}_{1}\overline{\psi_{J}}$ can be expressed as the sum of four positive functions
$D^{-1}_{1}\overline{\psi_{J}}=f_{0}-f_{2}+i(f_{1}-f_{3})=\sum_{k=0}^{3}i^{k}f_{k}$. Hence, 
applying the Mean Value Theorem for integrals to each positive function $f_{k}$ with respect to the variable $x' \in \R^{d-1}$, we obtain
\begin{align*}
\langle \phi_{I}, \psi_{J}\rangle 
%&=- \sum_{k=0}^{3}i^{k}\int_{\mathbb R^d} \partial_{1}\phi_{I}(x)f_{k}(x)dx
%\\
&=- \sum_{k=0}^{3}i^{k}\int_{\mathbb R}\partial_{1}\phi_{I}\big(x_{1},g_{k}(x_{1})\big)
\bigg(\int_{\mathbb R^{d-1}}f_{k}(x)dx'\bigg)dx_{1},
\end{align*}
where the functions $g_{k}$   denote the dependence of all coordinates from $x_{1}$.
Hence, 
$$
|\langle \phi_{I}, \psi_{J}\rangle |
\leq \int_{\mathbb R}\sup_{k}|\partial_{1}\phi_{I}(x_{1},g_{k}(x_{1}))|
\bigg(\sum_{k=0}^{3}\int_{\mathbb R^{d-1}}f_{k}(x)dx'\bigg)dx_{1}.
%g_{i}(x_{1})dx_{1}
%=\langle \varphi_{I}, D^{-1}\psi_{I}\rangle 
$$
%and 
%we denote $\varphi_{I}(t)=\partial_{1}\phi_{I}(t,g_{i}(t))$.
%then, $g'(x)=\langle \nabla \phi_{I},(1,\nabla f)\rangle $
Since
\begin{align*}
D^{-1}\overline{\psi_{J}}(t) &=\int_{\mathbb R^{d-1}}\hspace{-.3cm}D^{-1}_{1}\overline{\psi_{J}}(t, x')dx'
=\sum_{k=0}^{3}i^{k}\int_{\mathbb R^{d-1}}\hspace{-.3cm}f_{k}(t,x')dx' 
=\sum_{k=0}^{3}i^{k}F_{k}(t),
\end{align*}
we have that
\begin{align*}
\sum_{k=0}^{3}\int_{\mathbb R^{d-1}}f_{k}(t,x')dx'
&\leq 2^{1/2}\bigg(\Big(\sum_{k \hskip2pt {\rm even}}F_{k}(t)\Big)^{2}
+\Big(\sum_{k \hskip2pt {\rm odd}}F_{k}(t)\Big)^{2}\bigg)^{1/2}
\\
&=2^{1/2}|D^{-1}\overline{\psi_{J}}(t)|.
\end{align*}
Therefore, we can write
\begin{equation}\label{finalform2bound}
	|\langle \phi_{I}, \psi_{J}\rangle_{L^2(\R^d)} | \lesssim \langle \sup_k |\partial_{1}\phi_{I}(t,g_{k}(t))|,  |D^{-1}(\psi_{J})(t)| \rangle_{L^2(\R)}.
\end{equation}

Now, on the one hand, we have by Definition \ref{adapted},
\begin{align*}
|\partial_{1}\phi_{I}(t,g_{k}(t))|
&\leq  \frac{C}{|I|^{\frac{1}{2}}\ell(I)}   \bigg(1+ \frac{|(t,g_{k}(t))-c(I)|_{\infty }}{\ell(I)}\bigg)^{-N}
\\
&\leq  \frac{C}{|I|^{\frac{1}{2}}\ell(I)^{\frac{1}{2}}} \frac{1}{\ell(I_{1})^{\frac{1}{2}}} \bigg(1+ \frac{|t-c(I_{1})|}{\ell(I_{1})}\bigg)^{-N}.
\end{align*}
This is the decay estimate in   Definition \ref{adapted} of a function being adapted to the interval $I_{1}$ with 
constant $C|I|^{-1/2}\ell(I)^{-1/2}$.
%, which is the only requirement for the inequality (\ref{1/2}).

On the other hand, 
to control the second factor in \eqref{finalform2bound}, we make the following computation: since $|\cdot |_{1}\leq d|\cdot |_{\infty }$, 
\begin{align*}
|D^{-1}(\overline{\psi_{J}})(t)|
&
%=|D^{-1}(\psi_{J})(t)|
%\leq \int_{te_{1}+H^{+}_{1}}|\psi_{J}(x)| \mathrm{d}x
%\\
%&
\leq \frac{C}{|J|^{\frac{1}{2}}} \int_{te_{1}+H^{+}_{1}} \bigg(1+\frac{|x-c(J)|_{\infty }}{\ell(J)}\bigg)^{-N}\mathrm{d}x
\\
&\leq  \frac{C}{|J|^{\frac{1}{2}}}\int_{te_{1}-c(J)+H^{+}_{1}}\bigg(1+\frac{|x|_{1}}{\ell(J)d}\bigg)^{-N}\mathrm{d}x
\\
&= \frac{C}{|J|^{\frac{1}{2}}}\int_{\mathbb R^{d-1}}\int_{t-c(J_{1})}^{\infty } \bigg(1+\frac{|x_{1}|+|x'|_{1}}{ \ell(J)d }\bigg)^{-N}\mathrm{d}x_{1}\mathrm{d}x'
%\\
%&\lesssim \frac{C}{|J|^{\frac{1}{2}}}\int_{\mathbb R^{d-1}}    \ell(J) d \bigg(1+  \frac{|t-c(J_{1})|+|x'|_{1}}{ \ell(J)d  } \bigg)^{-N+1} %\mathrm{d}x'
\\
&\lesssim \frac{C}{|J|^{\frac{1}{2}}}  \ell(J)^{d}d^{d}   \bigg(1+\frac{|t-c(J_{1})|}{\ell(J)d}\bigg)^{-N+d} 
\\
&
\leq C d^{N} |J|^{\frac{1}{2}} \bigg(d+ \frac{|t-c(J_{1})|}{\ell(J)}\bigg)^{-N+d}
\\
&\lesssim C |J|^{\frac{1}{2}}\ell(J)^{\frac{1}{2}}\frac{1}{\ell(J_{1})^{\frac{1}{2}}} 
\bigg(1+\frac{|t-c(J_{1})|}{\ell(J_{1})}\bigg)^{-N+d}.
\end{align*}
Here, we tacitly assumed that $t - c(J_1)>0$. If the opposite is true,  we use \eqref{the switch} in the first line of the argument to make the same calculation work.
We note that this is the decay estimate in   Definition \ref{adapted} of a function being adapted to the interval $I_{1}$ with constant $C|J|^{1/2}\ell(J)^{1/2}$.

Now, we combine the above two estimates.
Repeating the proof of \eqref{1/2} in the one-dimensional case for the expression in \eqref{finalform2bound}, starting with the splitting in \eqref{the start}, we obtain, for $N>d$, the bound
\begin{align*}
|\langle \phi_{I}, \psi_{J}\rangle |
&\lesssim C^{2}  \,  \bigg( \frac{ |J| \ell(J) }{ |I| \ell(I)} \bigg)^{\frac{1}{2}}  
\Big(\frac{\ell(J)}{\ell(I)}\Big)^{\frac{1}{2}}\bigg(1+ \frac{|c(I_{1})-c(J_{1})|}{  \ell(I)  }\bigg)^{-N+d}
\\
%$$
%=C|J||I|^{-\frac{1}{d}}|I|^{-\frac{1}{2}}|J|^{-1/2}|J|^{\frac{1}{d}}(1+ \ell(I)^{-1}|c(I)-c(J)|_{\infty })^{-N+d}
%$$
&= C^{2}\Big(\frac{|J|}{|I|}\Big)^{\frac{1}{2}+\frac{1}{d}}\bigg(1+  \frac{|c(I)-c(J)|_{\infty }}{  \ell(I) }\bigg)^{-N+d}.
\end{align*}
\end{proof}

\begin{lemma}\label{lowoscillation}
Let $I$ be a cube and $f$ be an integrable function supported
on $I$ with mean zero. For each dyadic cube $J$, let $\phi_{J}$ be a bump function adapted
to $J$ with constant $C>0$ and order 
%one and decay 
$N$. 

Then, for all dyadic cubes $J$ such that 
$\ell(J)\leq \ell(I)$, we have
\begin{equation}\label{L1control1}
|\langle f, \phi_{J} \rangle | |J|^{\frac{1}{2}}
\lesssim C\| f\|_{L^{1}(\mathbb R^{d})}
\Big( 1+\frac{|c(I)-c(J)|_{\infty }}{\ell(I)}\Big)^{-N},
\end{equation}
while for $\ell(I)\leq \ell(J)$, we get
\begin{equation}\label{L1control2}
|\langle f, \phi_{J} \rangle | |J|^{\frac{1}{2}}
\lesssim C\| f\|_{L^{1}(\mathbb R^{d})}\frac{\ell(I)}{\ell(J)}
\Big( 1+\frac{|c(I)-c(J)|_{\infty }}{\ell(J)}\Big)^{-N}.
\end{equation}

\end{lemma}

\begin{proof}

The proofs of both inequalities follow the pattern from the previous lemma with the required modifications to take advantage of the compact support of $f$. 

In order to prove \eqref{L1control1}, we divide the argument into two cases. 
%proceed as in the proof of inequality (\ref{1/2}), with some changes to take advantage of the compact support of $f$.
When $|c(I)-c(J)|_{\infty }\leq 2\ell(I)$, the inequality follows from   H\"older:
$$
|\langle f,\varphi_{J}\rangle |\leq \| f\|_{L^{1}(\mathbb R^{d})}\| \varphi_{J}\|_{L^{\infty }(\mathbb R^{d})}
\leq C\| f\|_{L^{1}(\mathbb R^{d})}\frac{1}{|J|^{1/2}}.
$$
%On the other hand, 
%$$
%|\langle \phi_{I},\varphi_{J}\rangle_{b^{2}} |
%\leq \| \phi_{I}\|_{\infty }\| \varphi_{J}\|_{\frac{q}{q-2}}\| b^{2}\|_{L^{\frac{q}{2}}(J)}
%\lesssim |I|^{-1/2}|J|^{-1/2+1-\frac{2}{q}}\| b\|_{L^{q}(J)}^{2}
%$$
%$$
%=\Big(\frac{|J|}{|I|}\Big)^{1/2}|J|^{2/q}\| b\|_{L^{q}(J)}^{2}
%\lesssim \Big(\frac{|J|}{|I|}\Big)^{1/2}B(J)^{2}
%$$
When $|c(I)-c(J)|_{\infty }>2\ell(I)$, we denote
by $c$ the midpoint between $c(I)$ and $c(J)$, and by $H\subset \mathbb R^{d}$ the hyperplane 
passing through $c$ and perpendicular to the line containing $c(I)$ and $c(J)$.
Let also $H_{J}$ be the half-space defined by the connected component of $\mathbb R^{d}\backslash H$ so that 
$c(J) \in H_{J}$. 
It can be readily checked that  $I\cap H_{J}=\emptyset$, whence
%Now we prove by contradiction that $I\cap H_{J}=\emptyset $. Let $x\in I\cap H_{J}$, which by definition satisfies $|x-c(I)|_{\infty }\leq \ell(I)/2$ and 
%$|x-c(J)|_{\infty }\leq |x-c(I)|_{\infty }$. Then, 
%$$
%|x-c|_{\infty }\leq \frac{|x-c(I)|_{\infty }}{2}+\frac{|x-c(J)|_{\infty }}{2}\leq \frac{\ell(I)}{2}
%$$
%On the other hand,
%$$
%|x-c|_{\infty }\geq |c-c(I)|_{\infty }-|x-c(I)|_{\infty }\geq \frac{|c(I)-c(J)|_{\infty }}{2}-\frac{\ell(I)}{2}>\frac{\ell(I)}{2}
%$$
%which proves the claim.  
  $\supp f\subset I\subset H_{J}^{c}$, and thus,
\begin{align*}
|\langle f, \phi_{J} \rangle | 
&\leq \int_{H_{J}^{c}}|f(x)| |\phi_{J}(x)|dx
%+\int_{H_{J}}|f(x)| |\phi_{J}(x)|dx
%$$
%$$
\\
&\leq  \| f\|_{L^{1}(\mathbb R^{d})}\| \phi_{J}\|_{L^{\infty }(H_{J}^{c})}
%+ \| f\|_{L^{\infty }(H_{J})}\| D\phi_{J}\|_{L^{1}(\mathbb R^{d})}
\\
&\leq \| f\|_{L^{1}(\mathbb R^{d})} \frac{C}{|J|^{1/2}}\bigg( 1+\frac{|c(I)-c(J)|_{\infty }}{\ell(J)}\bigg)^{-N},
\end{align*}
which is smaller than the bound in \eqref{L1control1} since $\ell(J)\leq \ell(I)$.
%since $\supp f\subset H_{J}^{c}$.

To prove  \eqref{L1control2}, we divide in two similar cases. 
%follows the pattern from the previous lemma with the required modifications to take advantage of the compact support of $f$. 
We first assume that $|c(I)-c(J)|_{\infty }\geq 2\ell(J)$. 
Let $c = (c_1,c') \in \R \times \R^{d-1}$ be the midpoint between $c(I)$ and $c(J)$ and let $H$ and $H_{J}$ be as before. 
Since now $\ell(I)\leq \ell(J)$, we have again that 
%Let 
%$H\subset \mathbb R^{d}$ be the hyperplane 
%passing through $c$ and perpendicular to the line containing $c(I)$ and $c(J)$.
%Let also $H_{J}$ be the half-space defined by the connected component of $\mathbb R^{d}\backslash H$ so that 
$c(J) \in H_{J}$ and $\supp f\subset H_{J}^{c}$. 

As in the previous lemma, we assume without loss of generality that 
$|c(I)-c(J)|_{\infty }=|c(I_{1})-c(J_{1})|$.
Then, we consider again the operator
$$
D^{-1}(f)(t)
%=\int_{\mathbb R^{d-1}}\int_{-\infty }^{t}\phi_{J}(y)dy
=\int_{te_{1}+H^{-}_{1}}f(x)dx
=-\int_{te_{1}+H_{1}^{+}}f(x)dx,
$$
where $H^{-}_{1}=\{ x\in \mathbb R^{d}: x_{1}\leq 0\}$ and 
$H_{1}^{+}=\{ x\in \mathbb R^{d}: 0\leq x_{1}\} $, due to the vanishing mean  of $f$. Moreover, 
the support of $D^{-1}f$ is included in $I_{1}$, which is, in turn, included in $(-\infty ,c_{1})$.

From the computations developed in the proof of (\ref{1/2+1/d}), we have that
$$
\langle f, \phi_{I}\rangle 
= - \sum_{k=0}^{3}i^{k}\int_{\mathbb R} \bigg(\int_{\mathbb R^{d-1}}
f_{k}(x_{1},x')dx' \bigg)\partial_{1}\overline{\phi_{J}}\big(x_{1},g_{k}(x_{1})\big)dx_{1},
$$
where $D^{-1}_{1}f
%=f_{0}-f_{2}+i(f_{1}-f_{3})
=\sum_{k=0}^{3}i^{k}f_{k}$, with $f_{k}$ positive functions and  
the functions $g_{k}$  denote the dependence of all coordinates from $x_{1}$.
Now, as in the proof of Lemma \ref{bumpsinteracting}, we have the inequalities
$$
|\partial_{1}\overline{\phi_{J}}(t,g_{k}(t))|
\leq \frac{C}{|J|^{1/2}\ell(J)} \bigg(1+\frac{|t-c(J_{1})|}{ \ell(J) }\bigg)^{-N},
$$
%and so, $\varphi_{J}$ is a bump function adapted to $J$ with order zero and constant $C\ell(J)$. 
and
%On the other hand, we have as before that
$$
\sum_{k=0}^{3}\int_{\mathbb R^{d-1}}f_{k}(t,x')dx'
\leq 2^{1/2}|D^{-1}f(t)|.
$$
Let  $\varphi_{J}(t)=\sup_{k}|\partial_{1}\overline{\phi_{J}}(t,g_{k}(t))|$. Then,  
\begin{align*}
|\langle f, \phi_{J} \rangle | 
&\lesssim \int_{-\infty }^{c_1}|D^{-1}f(t)| \varphi_{J}(t)dt
%+\int_{c_1}^{\infty }|D^{-1}f(t)| |\varphi_{J}(t)|dt
%\\
%& 
\leq  \| D^{-1}f\|_{L^{1}(-\infty ,c_{1})}\| \varphi_{J}\|_{L^{\infty }(-\infty ,c_{1})}
%+\| D^{-1}f\|_{L^{\infty }(c_{1},\infty )}\| \varphi_{J}\|_{L^{1}(c_{1},\infty )}
\\
& \leq \| D^{-1}f\|_{L^{1}(\mathbb R)} \frac{C}{|J|^{1/2}\ell(J)} \bigg( 1+\frac{|c(I_{1})-c(J_{1})|}{\ell(J)}\bigg)^{-N}.
\end{align*}
Now, from the bound $\| D^{-1}f\|_{L^{1}(\mathbb R)}\leq \ell(I_{1})|\| f\|_{L^{1}(\mathbb R^{d})}$ and 
the assumption about the first coordinate, we obtain the bound stated in \eqref{L1control2}.
%$$
%|\langle f, \phi_{J} \rangle | |J|^{1/2}
%\leq C\| f\|_{L^{1}(\mathbb R^{d})}\frac{\ell(I)}{\ell(J)}\bigg( 1+\frac{|c(I)-c(J)|_{\infty }}{\ell(J)}\bigg)^{-N}.
%$$

Finally, when $|c(I)-c(J)|\leq 2\ell(J)$, we use the easier estimate
\begin{align*}
|\langle f,\phi_{J}\rangle |
& \lesssim \int_{-\infty }^{c_1}|D^{-1}f(t)| \varphi_{J}(t)dt
%=|\langle D^{-1}f, \varphi_{J} \rangle |
%\\&
\leq \| D^{-1}f\|_{L^{1}(\mathbb R)}\| \varphi_{J}\|_{L^{\infty }(\mathbb R)}\\
&\leq \ell(I)\| f\|_{L^{1}(\mathbb R^{d})}C|J|^{-1/2}\ell(J)^{-1}.
\end{align*}
This ends the proof.
\end{proof}

\section{Compactness from $L^{1}(\mathbb R^{d})$ into $L^{1,\infty}(\mathbb R^{d})$} \label{L1weak}

In this section, we state and prove our main result. 

\begin{theorem}\label{theoweakL1}
Let $T$ be a linear operator associated with a standard Calder\'on-Zygmund kernel.
Then, $T$ can be extended to a compact operator from $L^{1}(\mathbb R^{d})$ into $L^{1,\infty }(\mathbb R^{d})$
if and only if  it is associated with a compact Calder\'on-Zygmund kernel satisfying 
the weak compactness condition and the cancellation conditions 
$T(1), T^{*}(1)\in \CMO(\mathbb R^{d})$.
\end{theorem}

\begin{remark}\label{fromptoendpoint} If $T$ is a linear operator with a standard Calder\'on-Zygmund kernel which can be extended compactly on $L^{p}(\mathbb R^{d})$ for any $1<p<\infty $ then, we know by 
Theorem \ref{Mainresult}
%\cite{PPV} 
that $T$ satisfies the same three hypotheses for compactness of Theorem \ref{theoweakL1}
%$$
%| \{ x : P_{M}^{\perp }Sf(x)>\lambda \} | \leq C\epsilon \| f\|_1 \lambda^{-1}  (?)
%$$
and so, it can also be extended as a compact operator from $L^{1}(\mathbb R^{d})$ into $L^{1,\infty }(\mathbb R^{d})$.
\end{remark}

We first justify the converse, which is
essentially a consequence of Theorem \ref{fromweak2strong} and compact real interpolation.  
We assume that $T$ is compact from $L^{1}(\mathbb R^{d})$ into $L^{1,\infty }(\mathbb R^{d})$. Since $T$ is bounded between the same spaces, by 
Theorem \ref{fromweak2strong},  we have that $T$ is also bounded on, say, $L^{4}(\mathbb R^{d})$. Then, by 
the interpolation Theorem \ref{compactinterpolation} with $\theta =\frac{1-1/2}{1-1/4}$, we obtain that 
$T$ is compact on $L^{2}(\mathbb R^{d})$. Now, the reverse implication of Theorem \ref{Mainresult} implies the required hypotheses: $T$ has a compact Calder\'on-Zygmund kernel and it satisfies both
the weak compactness condition and the cancellation conditions
$T(1), T^{*}(1) \in \CMO(\mathbb R^{d})$.

The remainder of the paper is devoted to show sufficiency. As in the study of boundedness, the proof is split into two cases: first a special case, when  extra cancellation properties are assumed (Proposition \ref{weakL1}); and second, the general case, which is dealt with by proving compactness of paraproducts (Proposition \ref{paraproducts1}).

\begin{proposition}\label{weakL1}
Let $T$ be a linear operator associated with a compact Calder\'on-Zygmund kernel 
satisfying 
the weak compactness condition and the cancellation conditions 
%which is compact on $L^{p}(\mathbb R)$ and satisfying 
$T(1)=0$ and $T^{*}(1)=0$.
Then, $T$ is compact from $L^{1}(\mathbb R^{d})$ into $L^{1,\infty }(\mathbb R^{d})$.
\end{proposition}
\begin{proposition}\label{paraproducts1}
Given a function $b$ $\in \CMO(\mathbb R^{d})$, there exists a linear operator $T_b$ associated with a compact Calder\'on-Zygmund kernel 
such that $T_b$  and $T_b^{*}$ are compact from $L^{1}(\mathbb R^{d})$ into $L^{1,\infty }(\mathbb R^{d})$ and  satisfy
$
\langle T_b(1),g\rangle =\langle b,g\rangle
$
and
$
\langle T_b(f),1\rangle =0
$, 
for all $f,g\in {\mathcal S}(\mathbb R^{d})$.
\end{proposition}

We now remind the reader how to deduce Theorem \ref{theoweakL1} from these propositions.  
The argument follows the well-known scheme 
provided in the proof of the classical $T(1)$ theorem. 
Namely, when $b_1=T(1)$, $b_2=T^{*}(1)$ are  functions in $\CMO(\mathbb R^{d})$, we use Proposition \ref{paraproducts1} to construct the paraproduct operators $T_{b_{i}}$. As proved in \cite{V1}, they have compact 
Calder\'on-Zymund kernels,  
are compact operators on $L^{2}(\mathbb R^{d})$ (and thus, they satisfy the weak compactness condition), 
and satisfy 
$T_{b_1}(1)=b_1$, $T_{b_2}(1)=b_2$ and $T_{b_1}^{*}(1)=T_{b_2}^{*}(1)=0$.
It now follows from Proposition \ref{weakL1} that the operator
$$
%\tilde{T}=
T-T_{b_1}-T_{b_2}^{*}
$$ 
is  compact from $L^{1}(\mathbb R^{d})$ into $L^{1,\infty}(\R^d)$. Finally, after proving that $T_{b_1}$ and $T_{b_2}^{*}$ are compact from $L^{1}(\mathbb R^{d})$ into $L^{1,\infty}(\R^d)$, we deduce that  
the same holds for the initial operator $T$.

\subsection{Proof of Proposition \ref{weakL1}}
Let $(\psi_{I}^{i})_{I\in {\mathcal D},i=1,\ldots, 2^{d}-1}$ be an orthogonal wavelet basis of $L^{2}(\mathbb R^{d})$ such that every function 
$\psi_{I}^{i}$ is adapted to a dyadic cube $I$ with constant $C>0$ and order $N$. We denote by  $P_{M}$ the lagom  projection of Definition \ref{lagom} associated with $(\psi_{I}^{i})_{I\in {\mathcal D},i=1,\ldots, 2^{d}-1}$. Since the index $i\in \{1,\ldots, 2^{d}-1\}$ and the dual wavelet play no significant role in the proof, in order to simplify notation, we will write the wavelet decomposition in $L^{2}(\mathbb R^{d})$ simply as 
$f=\sum_{I\in \mathcal D}\langle f,\psi_{I}\rangle \psi_{I}$.

By the classical theory, we know that $T$ extends to a bounded operator 
on $L^{p}(\mathbb R^{d})$, and 
from 
$L^{1}(\mathbb R^{d})$ into $L^{1,\infty}(\mathbb R^{d})$. Therefore, for every $f\in \mathcal S(\mathbb R^{d})$, $Tf$ and also $P_{M}^{\perp }Tf$
are meaningful as functions in the intersection of $L^p(\R^d)$ and $L^{1,\infty}(\mathbb R^{d})$.

Fix $1<p<\infty $. By Theorem \ref{Mainresult}, we already know that $T$ extends to a compact
operator on $L^{p}(\mathbb R^{d})$. Hence, 
for every $\epsilon >0$, there exists an $M_{0}\in \mathbb N$ such that, for all $M>M_{0}$ and   $f\in {\mathcal S}(\mathbb R^{d})$, we have
\begin{equation}\label{pcompactness}
\| P^{\perp}_{M}Tf\|_{L^{p}(\mathbb R^{d})}\lesssim \epsilon \| f\|_{L^{p}(\mathbb R^{d})},
\end{equation}
where the implicit constant depends only on $T$ and $p$.

According to Remark \ref{resizzy}, it suffices to prove that 
for any given $\epsilon >0$ and its corresponding $M_{1}\in \mathbb N$, we have for all 
$M>M_{1}$ with $M2^{-M\delta }+M^{-\delta }<\epsilon $, 
\begin{equation} \label{the target man}
m( \{ x \in \mathbb R^{d}: |P_{2M}^{\perp }Tf(x)|>\lambda \} ) \lesssim \frac{\epsilon}{\lambda} \| f\|_{L^{1}(\mathbb R^{d})}
\end{equation}
for all $f\in {\mathcal S}(\mathbb R^{d})$ and all $\lambda >0$. 
The implicit constant is allowed to depend on $\delta >0$, the parameter of the compact Calder\'on-Zygmund kernel, and the constant given by the wavelet basis, but is to be independent of $\epsilon$, $f$ and $\lambda $. 
%By density, we can assume $f\in {\mathcal S}(\mathbb R)$.

To prove \eqref{the target man}, we perform a classical Calder\'on-Zygmund decomposition of $f$ at level 
$\epsilon^{-1}\lambda>0$. For this,  
we consider the collection $\cal I$ of maximal dyadic cubes $I$ with respect
to set inclusion such that
$$
\frac{1}{|I|}\int_{I}|f(x)|dx > \frac{\lambda}{\epsilon}.
$$
Let $E$ be the disjoint union of all $I\in {\cal I}$, which  satisfies 
$m(E) \leq \epsilon \lambda^{-1}\|f\|_{L^{1}(\mathbb R^{d})}$.
With this, we
define the usual Calder\'on-Zygmund decomposition $f=\tilde{g}+\tilde{b}$, where
$$
\tilde{g}=\sum_{I\in \cal I} m_{I}(f)\chi_I+f\chi_{E^{c}},
\hskip 30pt \tilde{b}=\sum_{I\in {\mathcal I}}f_{I}=\sum_{I\in \cal I} \big(f-m_I(f)\big)\chi_I,
$$
with $m_I(f)=|I|^{-1}\int_{I}f(x)dx$.

By standard arguments, it follows that $\norm{\tilde{g}}_{L^{\infty }(\mathbb R^{d})} \leq 2^{d}\lambda/\epsilon$, and moreover, that 
$\norm{\tilde{g}}_{L^{1}(\mathbb R^{d})} \leq \norm{f}_{L^{1}(\mathbb R^{d})}$.
From this, the inequality (\ref{pcompactness}), and the fact that $M>M_{1}>M_{0}$,  we get
\begin{align*}
\| P^{\perp}_{2M}T\tilde{g} \|_{L^{p}(\mathbb R^{d})}^{p}
&\lesssim \epsilon^{p} \| \tilde{g}\|_{L^{p}(\mathbb R^{d})}^{p}
\lesssim \epsilon^{p} \int_{\mathbb R^{d}} |\tilde{g}(x)|\frac{\lambda^{p-1}}{\epsilon^{p-1}}dx 
\\
&
\leq \epsilon \lambda^{p-1} 
\| f\|_{L^{1}(\mathbb R^{d})}.
\end{align*}
Whence,
\begin{align*}\label{function}
%\nonumber
m\big( \{ x \in \mathbb R^{d} : |P^{\perp}_{2M}T\tilde{g}(x)| > \lambda/2\} \big)
&\lesssim \frac{1}{\lambda^{p}} \| P^{\perp}_{2M}T\tilde{g}\|_{L^{p}(\mathbb R^{d})}^{p}
\lesssim \frac{\epsilon }{\lambda} \| f\|_{L^{1}(\mathbb R^{d})}.
\end{align*}

Now we need to prove the same estimate for $\tilde{b}$.
To do so, we
define $\tilde{E}$ as
the union of all cubes $10I$ with $I\in {\cal I}$. 
%We  also define $F$ as the union of all $J$ such that the corresponding
%$I$ satisfies $I\subset I'$ for some $I'\in {\cal I}$ and $\tilde{F}$ as the union of all $3J$ with $J\subset F$. 
Writing $\R^d = \tilde{E} \cup \tilde{E}^C$, yields
$$
m( \{ x \in \mathbb R^{d} : |P^{\perp}_{2M}T\tilde{b}(x)| > \lambda/2\} )
\lesssim m(\tilde{E})+\frac{1}{\lambda}\| P^{\perp}_{2M}T\tilde{b}\|_{L^1(\tilde{E}^C)}.
$$
Since $m(\tilde{E})\lesssim \epsilon \lambda^{-1} \| f\|_{L^{1}(\mathbb R^{d})}$, 
it remains to show that
$$
\| P^{\perp}_{2M}T\tilde{b} \|_{L^1(\tilde{E}^C)} \lesssim \epsilon \| f\|_{L^{1}(\mathbb R^{d})}.
$$
To prove this, it suffices to show that for each $I\in {\cal I}$, we have
\begin{equation}\label{eachI}
\| P^{\perp}_{2M}Tf_{I} \|_{L^1(\tilde{E}^C)} \lesssim \epsilon \| f_{I}\|_{L^{1}(\mathbb R^{d})}.
%\lambda |I|
\end{equation}
Indeed, by sub-linearity, this would imply
\begin{align*}
\| P^{\perp}_{2M}T\tilde{b} \|_{L^1(\tilde{E}^C)} 
%\leq \sum_{I\in {\mathcal I}}\| P^{\perp}_{M}Tf_{I} \|_{L^1(\mathbb R \setminus \tilde{E})} 
&\lesssim \epsilon \sum_{I\in {\mathcal I}}\| f_{I}\|_{L^{1}(\mathbb R^{d})} 
%\lesssim \epsilon \sum_{I\in {\mathcal I}}\int_{I}|f(x)|dx
\leq \epsilon \| f\|_{L^{1}(\mathbb R^{d})},
\end{align*}
whence, 
$$
m(\{ x\in \mathbb R^{d} : |P^{\perp}_{2M}T\tilde{b}(x)| > \lambda/2\} ) \lesssim \frac{\epsilon}{\lambda} \| f\|_{L^{1}(\mathbb R^{d})},
%+C\epsilon \| f\|_{L^{1}(\mathbb R)}\lambda^{-1}
$$
which would complete the proof.

The remainder of the proof therefore deals with  obtaining (\ref{eachI}). 
%By the proof of boundedness of $T$ from 
%$L^{1}(\mathbb R^{d})$ into $L^{1,\infty}(\mathbb R^{d})$ (see \cite{ST}), we already know that 
%$$
%\| Tf_{I} \|_{L^1(\tilde{E}^C)}\lesssim \| f_{I}\|_{L^{1}(\mathbb R^{d})}.
%$$
%Moreover, by the definition of $P_{2M}^{\perp}$, we have  
%\begin{align*}
%\| P_{2M}^{\perp}Tf_{I} \|_{L^1(\tilde{E}^C)}
%&\lesssim \| Tf_{I} \|_{L^1(\tilde{E}^C)} 
%+\| P_{2M}Tf_{I} \|_{L^1(\tilde{E}^C)}
%\\
%&\lesssim \| f_{I}\|_{L^{1}(\mathbb R^{d})}
%+\sum_{I\in {\mathcal D}_{M}}|\langle Tf_{I},\psi_{J}\rangle |\| \psi_{J}\|_{L^{1}(\mathbb R^{d})},
%\end{align*}
%which is a finite quantity since the sum is finite. Notice that, 
To this end, since 
%Since $f\in {\mathcal S}(\mathbb R^{d})$, we have 
$f_{I}=(f-m_I(f))\chi_I$ and apply % \in L^{2}(\mathbb R^{d})$ and apply  %so, 
%by the continuity of $P^{\perp}_{2M}$ and $T$ on $L^{2}(\mathbb R^{d})$, we also have
%$P^{\perp}_{2M}Tf_{I}\in L^{2}(\mathbb R^{d})$.
%Therefore, we can apply 
Fatou's Lemma to obtain 
$$
\| P^{\perp}_{2M}Tf_{I} \|_{L^1(\tilde{E}^{C})}
\leq \liminf_{R\rightarrow \infty }
\| P^{\perp}_{2M}Tf_{I} \|_{L^1(\tilde{E}^{C}\cap [-R,R]^{d})}.
$$
%To estimate the last quantity, we observe that by the continuity of $P^{\perp}_{2M}$ and $T$ on $L^{2}(\mathbb R^{d})$, we also have 
To estimate the last quantity, it suffices, by duality,  to check that for all $g\in {\mathcal S}(\mathbb R^{d})$  in the unit ball of $L^{\infty }(\mathbb R^{d})$ with compact 
support in $K_{R}=\tilde{E}^{C}\cap [-R,R]^{d}$,
 we have  
\begin{equation}\label{compactnessbyduality}
|\langle P^{\perp}_{2M}Tf_{I}, g\rangle |\lesssim \epsilon \| f_{I}\|_{L^{1}(\mathbb R^{d})}.
%\|  g\|_{L^{\infty }(\mathbb R^{d})}.
%\lambda |I|
\end{equation}
We now justify this claim. Observe that since $f_I \in L^{2}(\mathbb R^{d})$, it follows by the continuity of $T$ and  $P^{\perp}_{2M}$ on $L^{2}(\mathbb R^{d})$
that we also have  $P^{\perp}_{2M}Tf_{I}\in L^{2}(\mathbb R^{d})$. % Since 
%$f_{I}\in L^{2}(\mathbb R^{d})$, by the continuity of $P^{\perp}_{2M}$ and $T$ on $L^{2}(\mathbb R^{d})$, we also have
%$P^{\perp}_{2M}Tf_{I}\in L^{2}(\mathbb R^{d})$, 
%Now, 
Hence, the function
$h=\sign(P^{\perp}_{2M}Tf_{I})\chi_{K_{R}}$  
% Then, by a corollary of Lusin's Theorem, the function $h$ 
can be approximated in the norm of $L^{2}(K_{R})$ by a function $g\in {\mathcal S}(\mathbb R^{d})$ with compact support in $K_R$ such that $\| g\|_{L^{\infty }(\mathbb R^{d})}\leq \| h\|_{L^{\infty }(\mathbb R^{d})}=1$. 
%Note that   $g$ belongs to the unit ball of $L^{\infty }(\mathbb R^{d})$ and has compact support in $K_{R}$.
 With this, we have 
\begin{align*}
\| P^{\perp}_{2M}Tf_{I} \|_{L^1(K_{R})}
%&=|\langle P^{\perp}_{2M}Tf_{I}, h\rangle |\\
&\leq |\langle P^{\perp}_{2M}Tf_{I},g\rangle |
+ |\langle P^{\perp}_{2M}Tf_{I},h-g\rangle |,
\end{align*}
%Now, using again the boundedness of $P^{\perp}_{2M}$ and $T$, 
where the last term can be bounded by  a constant times
$$
\|P^{\perp}_{2M}Tf_{I}\|_{L^{2}(\mathbb R^{d})}
\|h-g\|_{L^{2}(K_{R})}
\lesssim \|f_{I}\|_{L^{2}(\mathbb R^{d})}\epsilon\frac{\|f_{I}\|_{L^{1}(\mathbb R^{d})}}{1+\|f_{I}\|_{L^{2}(\mathbb R^{d})}}.
$$
This ends the desired  justification. 

We work now to obtain (\ref{compactnessbyduality}). 
%Let $(\psi_{I})_{I\in {\mathcal D}}$ be the orthogonal wavelet basis of $L^{2}(\mathbb R^{d})$ defining $P_{M}$. 
We start by justifying the equality 
\begin{equation}\label{waveletsrepre}
\langle P^{\perp}_{2M}Tf_{I}, g\rangle 
=\sum_{J\in {\mathcal D}}\sum_{K\in {\mathcal D}_{2M}^{c}}\langle f_{I}, \psi_{J}\rangle 
\langle g,\psi_{K}\rangle \langle T\psi_{J}, \psi_{K}\rangle,
\end{equation}
for functions $g$ as described above.
Since
$g\in \mathcal S(\mathbb R^{d})$,
%\subset L^{\infty }(\mathbb R^{d})\cap L^{2}(\mathbb R^{d})$
we have that $P^{\perp}_{2M}g=g-\sum_{K\in {\mathcal D}_{2M}}\langle g,\psi_{K}\rangle \psi_{K}$ is a well defined bounded smooth function. Therefore, we can give sense to
$
\langle P^{\perp}_{2M}Tf_{I}, g\rangle 
=\langle Tf_{I}, P^{\perp}_{2M}g\rangle.
$
%where $T^{*}(P^{\perp}_{2M}g)$ has full sense due to the properties of a Calder\'on-Zygmund kernel. 
%
Moreover, since 
%$f\in {\mathcal S}(\mathbb R^{d})$, we have that 
$f_{I}\in L^{2}(\mathbb R^{d})$, we can write  
$f_{I}=\sum_{J\in {\mathcal D}}\langle f_{I},\psi_{J}\rangle \psi_{J}$ with convergence in 
$L^{2}(\mathbb R^{d})$.  
Also $g\in L^{2}(\mathbb R^{d})$ and so, according to Definition \ref{lagom}, 
%$P^{\perp}_{2M}g\in L^{2}(\mathbb R)$. This way, we have 
we have 
$P^{\perp}_{2M}g=\sum_{K\in {\mathcal D}_{2M}^{c}}\langle g,\psi_{K}\rangle \psi_{K}$ with convergence also in $L^{2}(\mathbb R^{d})$. 
We now write for all $M',M''>2M$,
\begin{multline*}
%\lim_{{\cal M}_{1}\rightarrow \infty }\lim_{{\cal M}_{2}\rightarrow \infty }
\big|\langle Tf_{I}, P^{\perp}_{2M}g\rangle -\sum_{J\in {\mathcal D_{M'}}}
\sum_{K\in {\mathcal D}_{M''}\backslash {\mathcal D}_{2M}}
\langle f_{I}, \psi_{J}\rangle 
\langle g,\psi_{K}\rangle \langle T\psi_{J}, \psi_{K}\rangle \big|
\\ \leq  \| Tf_{I}\|_{L^{2}(\mathbb R^{d})}
\Big\| P^{\perp}_{2M}g-\sum_{K\in {\mathcal D}_{M''}\backslash {\mathcal D}_{2M}}
\langle g,\psi_{K}\rangle\psi_{K}\Big\|_{L^{2}(\mathbb R^{d})}
\\
+
\| g\|_{L^{2}(\mathbb R^{d})} \Big\| T\big(f_{I}-\sum_{J\in {\mathcal D}_{M'}}\langle f_{I}, \psi_{J}\rangle \psi_{J}\big)\Big\|_{L^{2}(\mathbb R^{d})}
\end{multline*}
By all the stated relationships and the continuity of $T$ on $L^{2}(\mathbb R^{d})$,
both terms in previous inequality tend to zero when $M', M''$ tend to infinity. 
This justifies the equality (\ref{waveletsrepre}).

%Having established \eqref{waveletsrepre}, 
Then, it follows from the triangle inequality that 
%by formula (\ref{}), 
\begin{align*}
%\label{starting}
\nonumber
|\langle P^{\perp}_{2M}Tf_{I}, g\rangle |
&=|\langle Tf_{I}, P^{\perp}_{2M}g\rangle |\\
%=|\sum_{J\in {\mathcal D}}\sum_{K\in {\mathcal D}_{2M}^{c}}\langle f_{I}, \psi_{J}\rangle 
%\langle g,\psi_{K}\rangle \langle T\psi_{J}, \psi_{K}\rangle |
%$$
%$$
&\leq \sum_{J\in {\mathcal D}}\sum_{K\in {\mathcal D}_{2M}^{c}}|\langle f_{I}, \psi_{J}\rangle |
|\langle g,\psi_{K}\rangle | | \langle T\psi_{J}, \psi_{K}\rangle |.
\end{align*} 
Now, for any given $\epsilon>0$ and $M_{T,\epsilon}\in \mathbb N$, we have
by Proposition \ref{symmetricspecialcancellation}, 
\begin{equation} \label{whatsup}
|\langle T\psi_J,\psi_K\rangle |
\lesssim \frac{\ecc(J,K)^{\frac{1}{2}+\frac{\delta}{d}}}{\rdist(J,K)^{d+\delta } }
\Big( F(J_{1},\ldots ,J_{6}; M_{T,\epsilon })+\epsilon \Big),
%L(|I_{i}|)S(|I_{i}|)D(\rdist(I_{i},K_{M_{1}}))
\end{equation}
where we wrote the parameter $\delta'$ simply as $\delta $, 
$J_{1}=J$, $J_{2}=K$, $J_{3}=\langle J,K\rangle$, $J_{4}=\lambda_{1}\tilde{K}_{max}$, 
$J_{5}=\lambda_{2}\tilde{K}_{max}$ and  $J_{6}=\lambda_{2}K_{min}$, with parameters $\lambda_{1},\lambda_{2}\geq 1$ explicitly 
stated in the proposition.

To further simplify notation, we write the last factor as $F(J_{i})+\epsilon $. 
Applying \eqref{whatsup}, we get
\begin{equation*}
|\langle P^{\perp}_{2M}Tf_{I}, g\rangle |
\lesssim 
%\| g\|_{L^{\infty }(\mathbb R)}
%\hspace{-.4cm}
\sum_{J\in {\mathcal D}}\sum_{K\in {\mathcal D}_{2M}^{c}}
\hspace{-.2cm}|\langle f_{I}, \psi_{J}\rangle | |\langle g,\psi_{K}\rangle | 
%\\
%|K|^{1/2}
\frac{\ecc(J,K)^{\frac{1}{2}+\frac{\delta}{d}}}{\rdist(J,K)^{d+\delta }} \big(F(J_{i})+\epsilon\big).
\end{equation*}

Now, we parametrise both sums according to the eccentricities and relative distances: first of $J$ with respect to the fixed cube $I$ and, later, of $K$ with respect each cube $J$. 
To this end, for every $k\in \mathbb Z$ and $m\in \mathbb N$, $m\geq 1$, we define the family
$$
I_{k,m}=\{ J\in {\mathcal D}:\ell(I)=2^{k}\ell(J), m\leq \rdist(I,J)< m+1 \}.
$$
We note that
the cardinality of $I_{k,m}$ is $2^{\max(k,0)d}2d(2m)^{d-1}$.

In the same way, for every $e\in \mathbb Z$ and $n\in \mathbb N$, $n\geq 1$, and every given cube $J\in I_{k,m}$,
we define the family
$$
J_{e,n}=\{ K\in {\mathcal D}:\ell(J)=2^{e}\ell(K), n\leq \rdist(J,K)< n+1 \}
$$
whose the cardinality is  $2^{\max(e,0)d}2d (2n)^{d-1}$.
With all this, we have
\begin{align*} % \label{moduloin2}
%\nonumber 
|\langle P^{\perp}_{2M}Tf_{I}, g\rangle 
&|\lesssim 
\sum_{\tiny \begin{array}{l}k\in \mathbb Z\\ m\in \mathbb N\end{array}}
\sum_{\tiny \begin{array}{l}e\in \mathbb Z\\ n\in \mathbb N\end{array}}
\sum_{J\in I_{k,m}} 
\sum_{K\in J_{e,n}\cap {\mathcal D}_{2M}^{c}}
|\langle f_{I}, \psi_{J}\rangle | |\langle g,\psi_{K}\rangle | 
%|J|^{1/2} 2^{-e/2}
\\
%\nonumber
&  \qquad \qquad  \qquad \qquad  \cdot 2^{-|e|d(\frac{1}{2}+\frac{\delta}{d})}n^{-(d+\delta )} \, \, (F(J_{i})+\epsilon)
\\ 
& \lesssim 
\sum_{\tiny \begin{array}{l}e\in \mathbb Z\\ n\in \mathbb N\end{array}}
2^{-|e|(\frac{d}{2}+\delta)} n^{-(d +\delta)}
2^{\max(e,0)d}n^{d-1} 
%\sum_{J\in {\mathcal D}}
\\
%\nonumber
&\hskip 10pt \cdot 
\sum_{\tiny \begin{array}{l}k\in \mathbb Z\\ m\in \mathbb N\end{array}}
\sum_{J\in I_{k,m}} 
|\langle f_{I}, \psi_{J}\rangle | 
%|J|^{1/2} 2^{-e/2}
\sup_{K\in J_{e,n}\cap {\mathcal D}_{2M}^{c}}
|\langle g,\psi_{K}\rangle | (F(J_{i})+\epsilon).  
\end{align*}
A crude estimate yields
\begin{align}\label{usesupg}
\nonumber
|\langle g,\psi_{K}\rangle |&\leq \| g\|_{L^{\infty }(\mathbb R^{d})}\int_{\tilde{E}^{C}}|K|^{-\frac{1}{2}}
\Big(1+\frac{|x-c(K)|_{\infty }}{\ell(K)}\Big)^{-N}dx\\
&\lesssim
|K|^{1/2}\Big(1+\frac{\dist(\tilde{E}^{C},c(K))}{\ell(K)}\Big)^{-N}
\nonumber
\\
&= |J|^{1/2}2^{-ed/2}w(\tilde{E}^{C},K)^{-N},
\end{align}
where the expression $w(\tilde{E}^{C},K)$ is defined by the last equality.
Using this and $2^{-e/2}2^{\max(e,0)}=2^{|e|/2}$, the above inequality becomes:
\begin{multline} \label{themainguy}
	|\langle P^{\perp}_{2M}Tf_{I}, g\rangle |
 	\lesssim
	\sum_{e\in \mathbb Z} \sum_{n\in \mathbb N }2^{- |e|\delta }n^{-(1+\delta)} 
	\\ \cdot \sum_{k\in \mathbb Z}\sum_{m\in \mathbb N }\sum_{J\in I_{k,m}} 
|\langle f_{I}, \psi_{J}\rangle |  
|J|^{1/2} 
\hspace{-3mm}
 \sup_{\substack{K\in J_{e,n}\cap {\mathcal D}_{2M}^{c}}} 
 %\\ K \cap \tilde{E}^C \neq \emptyset}
 w(\tilde{E}^{C},K)^{-N}(F(J_{i})+\epsilon).
\end{multline}
To keep the notation simple, we take the supremum over the empty set to be zero (recall that the support of $g$ is contained in $\tilde{E}^C$). Also, observe that, even though it is hidden by our choice of notation, $F(J_i)$ depends
on both $J$ and $K$ and thus, depends on $k$, $m$, $e$ and $n$.

In order to estimate \eqref{themainguy}, we need to control the terms of the double inner sum. 
We split the argument into two cases, depending on the size of $F(J_i)$:

\begin{itemize}
	\item[$(I)$] $J_{i}\notin {\cal D}_{M}$ for all $i=1,\ldots ,6$.  
	\item[$(II)$] $J_i \in {\cal D}_{M}$ for some $i=1,\ldots ,6$.  
\end{itemize}

The point here is that in case $(I)$ we know, according to Lemma \ref{lesizzy},  that $F(J_i) < \epsilon$ and thus it suffices  to merely bound the sum by some constant. 
On the other hand, in case $(II)$ we only know that $F(J_i)$ is bounded and so, we need to use the size and location of the cubes $J$ and $K$ to deduce an estimate that depends on $\epsilon$. 

\vskip10pt
{\bf Proof of (I).} As already noted, in this first case we have 
$F(J_{i})
%L(|I_{i}|)S(|I_{i}|)D(\rdist(I_{i},K_{M_{1}}))
<\epsilon $.
We divide the study in two cases: $\ell(I)\leq \ell(J)$ and $\ell(I)>\ell(J)$, which correspond to the cases $k\leq 0$ and $k>0$ respectively.

The first case follows directly from Lemma \ref{lowoscillation}. 
Indeed, since $k\leq 0$ and $f_{I}$ is supported on $I$ with zero mean, by \eqref{L1control2} we have
%in Lemma \ref{lowoscillation}, we know that 
\begin{align}\label{L1}
\nonumber
|\langle f_{I}, \psi_{J}\rangle | |J|^{\frac{1}{2}}
&\lesssim \| f_{I}\|_{L^{1}(\mathbb R^{d})}\frac{\ell(I)}{\ell(J)} \bigg(1+\frac{|c(I)-c(J)|}{\ell(J)}\bigg)^{-N}
\\
&\lesssim \| f_{I}\|_{L^{1}(\mathbb R^{d})}2^{k}m^{-N}.
%\rdist_{J}(I,J)^{-N}
\end{align}
%$m\in \mathbb N$, 
Moreover, the cardinality of $J\in I_{k,m}$ is comparable to $m^{d-1}$ and so,
in light of \eqref{L1} and the inequality $F(J_i)<\epsilon$, the inequality  \eqref{themainguy}  becomes
\begin{align*} % \label{sumfirstcase}
	|\langle P^{\perp}_{2M}Tf_{I}, g\rangle |
 	&\lesssim
	\epsilon \sum_{e\in \mathbb Z} \sum_{n\in \mathbb N }2^{- |e|\delta }n^{-(1+\delta)} 
	\sum_{k\leq 0} \sum_{m\in \mathbb N }2^k  \norm{f_I}_{L^1(\R^d)}m^{d-1-N} \\
%	&\lesssim
%	\epsilon  \norm{f_I}_{L^1(\R^d)}  \sum_{e\in \mathbb Z} \sum_{n\in \mathbb N }2^{- |e|\delta }n^{-(1+\delta)} \\
	&\lesssim
	\epsilon  \norm{f_I}_{L^1(\R^d)},
\end{align*}
for $N>d$, with the implicit constant depending exponentially on $d$.

In the second case, however, we need to be more careful. Now we have $\ell(J)< \ell(I)$, or equivalently  $k>0$, 
and we further divide in two more cases: $\ell(I)< \ell(K)$ and  $\ell(K)\leq \ell(I)$. In the first case, 
we use \eqref{L1control1} in Lemma \ref{lowoscillation}, and so,
\begin{align*}
%\label{L1}
%\nonumber
|\langle f_{I}, \psi_{J}\rangle | |J|^{\frac{1}{2}}
&\lesssim \| f_{I}\|_{L^{1}(\mathbb R^{d})}\bigg(1+\frac{|c(I)-c(J)|}{\ell(I)}\bigg)^{-N}
\\
&\lesssim \| f_{I}\|_{L^{1}(\mathbb R^{d})}m^{-N}.
%\rdist_{J}(I,J)^{-N}
\end{align*}
Moreover, we have that 
$\ell(I)<\ell(K)=2^{-(e+k)}\ell(I)$ which implies 
$0<k\leq -e$. Note that $e\leq 0$ in this situation since $\ell(J)\leq \ell(I)\leq \ell(K)$. 
Then, since  $F(J_i)<\epsilon$, 
%Lemma \ref{lowoscillation}
 the bound for the corresponding terms of
(\ref{themainguy}) becomes 
\begin{equation*}
% \label{kadabra}
%	|\langle P^{\perp}_{2M}Tf_{I}, g\rangle |
% 	&\lesssim
	\begin{aligned}
        \epsilon \sum_{e\in \mathbb Z} &\sum_{n\in \mathbb N } 2^{- |e|\delta }n^{-(1+\delta)} 
	\sum_{k=1}^{ |e|} \sum_{m\in \mathbb N }
|\langle f_{I}, \psi_{J}\rangle | 
|J|^{1/2} \\
&\lesssim
	\epsilon \sum_{e\in \mathbb Z} \sum_{n\in \mathbb N }2^{- |e|\delta }n^{-(1+\delta)} |e|\sum_{m\in \mathbb N } \| f_{I}\|_{L^{1}(\mathbb R^{d})}m^{-N} \\
	&\lesssim
	\epsilon \| f_{I}\|_{L^{1}(\mathbb R^{d})}.
	\end{aligned}
\end{equation*}

We now assume that $\ell(K)\leq \ell(I)$. By the definition of bump functions adapted to a cube, we have 
$$
|\langle f_{I}, \psi_{J}\rangle | |J|^{\frac{1}{2}} 
\leq C\int_{I} |f_{I}(x)|
\bigg(1+\frac{|x-c(J)|_{\infty }}{\ell(J)}\bigg)^{-N}dx.
$$
Then, for every fixed eccentricity parameter $0< k$ and every fixed $x\in I$,
we proceed to parametrise the cubes $J\in I_{k,m}$ associated with a fixed value of
$\ell(J)^{-1}|x-c(J)|_{\infty }$.
Since
$$
m\leq \rdist(I,J)\leq 1+\frac{|c(I)-c(J)|_{\infty }}{\ell(I)}
$$
and
$$
m+1> \rdist(I,J)\geq \frac{1}{2}\Big(1+\frac{|c(I)-c(J)|_{\infty }}{\ell(I)}\Big),
$$
we get $(m-1)\ell(I)\leq |c(I)-c(J)|_{\infty }\leq (2m+1)\ell(I)$.
This way, 
$$
|x-c(J)|_{\infty }\geq |c(I)-c(J)|_{\infty }-|x-c(I)|_{\infty }
\geq (m-\frac{3}{2})\ell(I)\geq \frac{m-1}{2}2^{k}\ell(J)
$$
and
$$
|x-c(J)|_{\infty }\leq (2m+\frac{3}{2})\ell(I)\leq 2(m+1)2^{k}\ell(J).
$$
Moreover, for every fixed integer $(m-1)2^{k-1}\leq r\leq 2^{k+1}(m+1)$, there are at most 
$d2^{d}(r+1)^{d-1}$
cubes $J\in I_{k,m}$ with  
$r\leq \ell(J)^{-1}|x-c(J)|_{\infty }<r+1$. 
With all this, we have
\begin{align}\label{inside0}
\nonumber
\sum_{J\in I_{k,m}} 
|\langle f_{I}, \psi_{J}\rangle | |J|^{\frac{1}{2}} 
&\lesssim \int |f_{I}(x)|\sum_{r=(m-1)2^{k-1}}^{(m+1)2^{k+1}}(1+r)^{d-1-N}dx
\\
&\lesssim \| f_{I}\|_{L^{1}(\mathbb R^{d})}(1+(m-1)2^{k-1})^{d-N}.
\end{align}

For this reason, we again divide the argument into two cases: $J\cap 3I= \emptyset$ and $J\cap \, 3I\neq \emptyset$. In the former case, we have 
$m>1$, and so,  
with the estimate \eqref{inside0} and the inequality  $F(J_i)<\epsilon$,
%given by Lemma \ref{lowoscillation}, 
the bound for the corresponding terms of
(\ref{themainguy}) becomes
\begin{equation*} 
%\label{abra}
	\begin{aligned} 
%	|\langle P^{\perp}_{2M}Tf_{I}, g\rangle |
 %	&\lesssim
	\epsilon \sum_{e\in \mathbb Z} & \sum_{n\in \mathbb N }  2^{- |e|\delta }n^{-(1+\delta)} 
	\sum_{k\geq 1}\sum_{m\geq 2 }\sum_{J\in I_{k,m}} 
|\langle f_{I}, \psi_{J}\rangle | 
|J|^{1/2} \\
&\lesssim
	\epsilon  \| f_{I}\|_{L^{1}(\mathbb R^{d})} \sum_{e\in \mathbb Z} \sum_{n\in \mathbb N }2^{- |e|\delta }n^{-(1+\delta)} 
	\sum_{k\geq 1} \sum_{m\geq 2}m^{d-N}2^{k(d-N)}\\
	&\lesssim
	\epsilon \| f_{I}\|_{L^{1}(\mathbb R^{d})}
	\end{aligned}
\end{equation*}
as long as $N>d+1$.

In the case 
$J\cap \, 3I\neq \emptyset$, we have $m=1$ and $|c(I)-c(J)|_{\infty }\leq 5\ell(I)/2$. 
%On the other hand, we also have that $K\cap \tilde{E}^{c}\neq \emptyset$, and so, 
To continue the analysis, we need to split into two further subcases: $K\cap 7I= \emptyset$ and 
$K\cap 7I\neq  \emptyset$.

In the former case,  we have that 
$|c(K)-c(I)|_{\infty }\geq 7\ell(I)/2$
and so, 
\begin{align*}
n+1 	>\rdist(J,K) 
	&\geq \frac{1}{2} \bigg(1+\frac{|c(J)-c(K)|_{\infty }}{\max\big(\ell(J),\ell(K)\big)}\bigg) \\
	&\geq \frac{1}{2} \bigg(1+ \frac{|c(K)-c(I)|_{\infty }-|c(I)-c(J)|_{\infty }}{\max\big(\ell(J),\ell(K)\big)} \bigg) \\
	&\geq 
%2^{-1}(1+\max(|J|,|K|)^{-1}(3|I|/2-|I|/2))
%=
	\frac{1}{2} \bigg(1+\frac{\ell(I)}{\max\big(\ell(J),\ell(K)\big)}\bigg).
\end{align*}
Since 
$\ell(J)=2^{-k}\ell(I)$ and $\ell(K)=2^{-e}\ell(J)=2^{-(e+k)}\ell(I)$, this yields
$$
n+1>\frac{1}{2}\big(1+2^{k}\min(1,2^{e})\big),
$$
whence $$1\leq 2^{k}\leq \frac{2n+1}{\min(1,2^{e})}\leq 3n(1+2^{-e}).$$

With the estimate \eqref{inside0} and the inequality  $F(J_i)<\epsilon$,
%given by Lemma \ref{lowoscillation}, 
the bound for the corresponding terms of
(\ref{themainguy}) becomes
\begin{equation*} 
%\label{abra}
	\begin{aligned} 
%	|\langle P^{\perp}_{2M}Tf_{I}, g\rangle |
 %	&\lesssim
	\epsilon \sum_{e\in \mathbb Z} & \sum_{n\in \mathbb N }  2^{- |e|\delta }n^{-(1+\delta)} 
	\hspace{-3mm} 
	\sum_{k=1   }^{ \log(3n(1+2^{-e}))}\sum_{J\in I_{k,1}} 
|\langle f_{I}, \psi_{J}\rangle | 
|J|^{1/2} 
\\
&\lesssim
	\epsilon  \| f_{I}\|_{L^{1}(\mathbb R^{d})} \sum_{e\in \mathbb Z} \sum_{n\in \mathbb N }2^{- |e|\delta }n^{-(1+\delta)} 
	\log\big(3n(1+2^{-e})\big) \\
	&\lesssim
	\epsilon \| f_{I}\|_{L^{1}(\mathbb R^{d})}
	\end{aligned}
\end{equation*}
as long as $N>d$. 
%with the implicit constant depending exponentially on the dimension. 

%$(I_2):$ 
On the other hand, when $K\cap 7I\neq \emptyset $ and $\ell(K)\leq \ell(I)$, we have $\dist(\tilde{E}^{C},c(K))\geq \ell(I)$. Moreover, 
$\ell(I)=2^{k+e}\ell(K)$ with $k+e\geq 0$. Then, 
$$
w(\tilde{E}^{C},K)
=1+\frac{\dist(\tilde{E}^{C},c(K))}{\ell(K)}\geq 1+2^{k+e}.
$$
With this, the estimate \eqref{inside0}, and the inequality  $F(J_i)<\epsilon$,
%given by Lemma \ref{lowoscillation}, 
the bound for the corresponding terms of
(\ref{themainguy}) becomes
\begin{equation*} 
%\label{abra}
	\begin{aligned} 
%	|\langle P^{\perp}_{2M}Tf_{I}, g\rangle |
 %	&\lesssim
	\epsilon \sum_{e\in \mathbb Z} & \sum_{n\in \mathbb N }  2^{- |e|\delta }n^{-(1+\delta)} 
	\sum_{\tiny\begin{array}{c}k\geq 1\\ k+e\geq 0\end{array}
	%k\geq \max(1,e)
	}\sum_{J\in I_{k,1}} 
|\langle f_{I}, \psi_{J}\rangle | |J|^{1/2} w(\tilde{E}^{C},K)^{-N}\\
%&\lesssim
%	\epsilon  \| f_{I}\|_{L^{1}(\mathbb R^{d})} \sum_{e\in \mathbb Z} \sum_{n\in \mathbb N }2^{- |e|\delta }n^{-(1+\delta)} 
%	\hspace{-7mm}
%	\sum_{k=1   }^{ \log(3n(1+2^{-e}))}
%	\sum_{r=0}^{2^{k+2}}(r+1)^{d-1-N} \\
&\lesssim
	\epsilon  \| f_{I}\|_{L^{1}(\mathbb R^{d})} \sum_{e\in \mathbb Z} \sum_{n\in \mathbb N }2^{- |e|\delta }n^{-(1+\delta)} 
	\sum_{\tiny\begin{array}{c}k\geq 1\\ k+e\geq 0\end{array}} 2^{-(k+e)N}\\
	&\lesssim
	\epsilon \| f_{I}\|_{L^{1}(\mathbb R^{d})}.
	\end{aligned}
\end{equation*}

Combining all the obtained estimates, 
we get the desired bound for \eqref{themainguy} under the assumption of case $(I)$.

\vskip10pt
{\bf Proof of (II).} As previously stated, in this case we use the size and location of the cubes $J$ and $K$ to deduce an estimate that depends on $\epsilon$. This leads to the following sub-cases:
%In order to estimate \eqref{the main guy}, we divide the study into six cases:
\vskip10pt
\noindent
\hspace{0mm}
\begin{minipage}{7cm}
\begin{enumerate}
\item[$(II_1)$] $J_1 = J\in {\cal D}_{M}$
%$I,\langle I\cup J\rangle ,\lambda_{1}\tilde{K}_{max}, \lambda_{2}\tilde{K}_{max}, \lambda_{2}K_{min} \notin {\cal D}_{M}$
\item[$(II_2)$] $J_2= K\in {\cal D}_{M}$
\item[$(II_3)$] $J_3= \langle J\cup K\rangle \in {\cal D}_{M}$
\end{enumerate}
\end{minipage}
\hspace{-1.8cm}
\begin{minipage}{8cm}
\begin{enumerate}
\item[$(II_4)$] $J\notin {\cal D}_{M}$ but $J_4= \lambda_{1}\tilde{K}_{max}\in {\cal D}_{M}$
\item[$(II_5)$] $J\notin {\cal D}_{M}$ but $J_5 =\lambda_{2}\tilde{K}_{max}\in {\cal D}_{M}$
\item[$(II_6)$] $J\notin {\cal D}_{M}$ but $J_6 = \lambda_{2}K_{min} \in {\cal D}_{M}$
\end{enumerate}
\end{minipage}
\vskip10pt
We can use the fact that $K\in J_{e,n}\cap {\mathcal D}_{2M}^{c}$ to immediately rule out the case $(II_2)$. 
We note that the property $K\notin {\cal D}_{2M}$ 
%is not used in proving $(I)$, but it will 
plays a crucial role in the remaining cases. 
We prove only the case $(II_1)$ since, as explained in more detail in \cite{V1}, all other cases can be dealt with by a similar reasoning.

\vskip10pt
(II$_1$):  We recall that the cubes $J$ and $K$ in the sum (\ref{themainguy}) satisfy $\ell(J)=2^{e}\ell(K)$ and 
$n\leq \rdist(J,K)<n+1$. 

By assumption, we have
$J\in {\cal D}_{M}$. That is,   
$2^{-M}\leq \ell(J)\leq 2^{M}$ and $\rdist(J,\mathbb B_{2^{M}})\leq M$.
Also, since $F$ is bounded, we have   $F(J_{i})+\epsilon \lesssim 1$. 
%and the second term can be dealt in the same way as we did in the previous case. Therefore, we will assume $\max_{i}F(I_{i})\leq 1$.  

Since $K\in {\mathcal D}_{2M}^{c}$, 
%as is the usual technique in \cite{PPV}, 
we separate the study into  three  cases: 

\hspace{2mm}
\begin{minipage}{1\textwidth}

\hspace{0mm}
\begin{itemize}
	\item[$(II_{1.1})$] $\ell(K)>2^{2M}$, 
	\item[$(II_{1.2})$] $\ell(K)<2^{-2M}$
	\item[$(II_{1.3})$] $2^{-2M}\leq \ell(K)\leq 2^{2M}$ with $\rdist(K,\mathbb B_{2^{2M}})>2M$. 
\end{itemize}
\hspace{0mm}
\end{minipage}

$(II_{1.1}):$ 
The inequalities $\ell(K)>2^{2M}$ and 
$2^{e}\ell(K)=\ell(J)\leq 2^{M}$ imply $2^{e}\leq 2^{M}\ell(K)^{-1}\leq 2^{-M}$ and so, 
$e\leq -M$. 
%Now, since $F_{\epsilon}(I_{i})\leq 1$, 

Using this and repeating the arguments from $(I)$, 
the inequality \eqref{themainguy} becomes
\begin{align*}
	|\langle P^{\perp}_{2M}Tf_{I}, g\rangle |
&\lesssim
	   \| f_{I}\|_{L^{1}(\mathbb R^{d})} \sum_{e \leq - M} \sum_{n\in \mathbb N }2^{- |e|\delta }n^{-(1+\delta)} 
	(|e| + \log n) \\
	&
%	\lesssim
%	\epsilon \| f_{I}\|_{L^{1}}
	\lesssim M2^{-M\delta }\| f_{I}\|_{L^{1}(\mathbb R^{d})}<\epsilon \| f_{I}\|_{L^{1}(\mathbb R^{d})},
\end{align*}
where the last inequality holds by the choice of $M$.

$(II_{1.2}):$ The case $\ell(K)<2^{-2M}$ is totally symmetrical with respect to the previous one, and amounts to changing $e\leq -M$ by $e\geq M$.
%in previous case. 

$(II_{1.3}):$ When  $2^{-2M}\leq \ell(K)\leq 2^{2M}$  and $\rdist(K,\mathbb B_{2^{2M}})\geq 2M$, we have 
%$\ell(K)=2^{s}$ with $-2M\leq s\leq 2M$ and, 
by Remark \ref{remark in section 2} that $|c(K)|_{\infty }\geq (2M-1)2^{2M}$.  
%and so 
%$|c(J)|=m|J|$ with $m\geq 2^{6M}|J|^{-1}$. 
Moreover, since $J\in {\mathcal D_{M}}$, by the same remark  we have 
%\ref{remark in section 2} 
$|c(J)|_{\infty }\leq (M-1)2^{M}$.
Then,
\begin{align*}
|c(J)-c(K)|_{\infty } &\geq |c(K)|_{\infty }-|c(J)|_{\infty } 
%\\ &\geq (2M-1)2^{2M}-(M-1)2^{M}\\ & 
\geq M2^{2M}.
\end{align*}
Furthermore, $\max(\ell(J),\ell(K))\leq 2^{2M}$ and so,
%for $e\leq 0$, we get
\begin{align*}
n+1 &> \rdist(J,K) 
%=\frac{\diam(J\cup K)}{\max(\ell(J),\ell(K))} \\ &
\geq \frac{|c(J)-c(K)|_{\infty }}{\max(\ell(J),\ell(K))}
\geq 
%2^{-2M}M2^{2M}=
M.
\end{align*}
Using this in combination with   the arguments from $(I)$, 
the inequality \eqref{themainguy} now becomes
\begin{align*}
	|\langle P^{\perp}_{2M}Tf_{I}, g\rangle |
&\lesssim
	   \| f_{I}\|_{L^{1}(\mathbb R^{d})} \sum_{e \in \mathbb{Z}} \sum_{n \geq  M-1}2^{- |e|\delta }n^{-(1+\delta)} 
	(|e| + \log n) \\
	&
%	\lesssim
%	\epsilon \| f_{I}\|_{L^{1}}
	\lesssim M^{-\delta }\| f_{I}\|_{L^{1}(\mathbb R^{d})}<\epsilon \| f_{I}\|_{L^{1}(\mathbb R^{d})},
\end{align*}
 again by the choice of $M$.

\qed

\subsection{Proof of Proposition \ref{paraproducts1}}
Let $(\psi_{I}^{i})_{I\in {\mathcal D},i=1,\ldots, 2^{d}-1}$ be an orthogonal wavelet basis of $L^{2}(\mathbb R^{d})$ such that every function 
$\psi_{I}^{i}$ is adapted to a dyadic cube $I$ with constant $C>0$ and order $N$.  As in the proof of Proposition \ref{weakL1}, we suppress the dependence on the  index $i\in \{1,\ldots, 2^{d}-1\}$.

We denote by $\varphi \in{\mathcal S}(\mathbb R^{d})$ a positive bump function adapted to $[-1/2,1/2]^{d}$ with order $N$ and constant $C>0$ such that 
%such that $\varphi (x)=c$ for all  $x\in [-1/4,1/4]$
 $\int_{\mathbb R^{d}}\varphi(x)dx=1$. 
In particular, we have that 
$0\leq \varphi (x)\leq C(1+|x|_{\infty })^{-N}$ and $|\partial_{i}\varphi(x)|\leq C(1+|x|_{\infty })^{-N}$ for all $i=1,\ldots ,d$.
%$\| \psi^2\|_{1}=1$.  
Let $(\varphi_{I})_{I\in {\mathcal D}}$  be
the family of bump functions defined by $\varphi_I(x)=\frac{1}{|I|}\varphi\Big(\frac{x-c(I)}{\ell(I)}\Big)$. 

Given a function $b\in \CMO(\mathbb R^{d})$, we define the linear operator $T_b$ by 
$$
\langle T_b f,g\rangle =\sum_{J\in {\mathcal D}}  \langle b, \psi_{J}\rangle \langle f, \varphi_{J}\rangle \langle g,\psi_{J}\rangle,
%\approx \langle \Pi_{b_1}(f_1)\otimes \Pi_{b_2}(f_2),g\rangle
%=\langle (\Pi \otimes \Pi )_{b}(f),g\rangle
$$
for all $f,g\in {\mathcal S}(\mathbb R^{d})$.
It was shown in \cite{V1} that 
$T_{b}$ and $T_{b}^{*}$ are associated with a compact Calder\'on-Zygmund kernel, are compact on $L^{p}(\mathbb R^{d})$ for every $1<p<\infty$, and they satisfy
$
\langle T_b(1),g\rangle =\langle b,g\rangle
$
and
$
\langle T_b(f),1\rangle =0
$. 

Now, we prove that $T_{b}$, $T_b^\ast$ are compact from $L^{1}(\mathbb R^{d})$ into $L^{1,\infty }(\mathbb R^{d})$. 
%.while 
%$T_{b}^{*}$ is not. 
%As we will see, the proof for the former operator is easier than for the latter one. 
%
To prove compactness of the former operator, we show first the equality $P_{M}^{\perp}T_b=T_{P_{M}^{\perp}b}$. 
Let $f,g\in {\mathcal D}(\mathbb R^{d})$. Since 
$P_{M}^{\perp }g=\sum_{J\in {\mathcal D}_{M}^{c}}  \langle g, \psi_{J}\rangle \psi_{J}$, 
%we have
%\begin{equation}\label{compactparaproduct}
\begin{align*}\label{easyparaproduct}
\nonumber
\langle P_{M}^{\perp}T_bf,g\rangle &=\langle T_bf_{I},P_{M}^{\perp}g\rangle 
=\sum_{J\in {\mathcal D}_{M}^{c}}  \langle b, \psi_{J}\rangle \langle f, \varphi_{J}\rangle \langle g,\psi_{J}\rangle
\\
%\end{equation}
%can rewrite previous expression as 
%\sum_{I} \langle b-P_{M}(b), \psi_{I}\rangle \langle f, \psi_{I}^2\rangle \langle g,\psi_{I}\rangle =
&=\sum_{J\in {\mathcal D}}   \langle P_{M}^{\perp }b, \psi_{J}\rangle \langle f, \varphi_{J}\rangle \langle g,\psi_{J}\rangle
=\langle T_{P_{M}^{\perp }b}(f),g\rangle ,
\end{align*}
where the second last equality holds because $b\in \CMO(\mathbb R^{d})$ and so, we also have
$
P_{M}^{\perp }b=\sum_{J\in {\mathcal D}_{M}^{c}}  \langle b, \psi_{J}\rangle \psi_{J}.
$

Moreover, $b\in \CMO(\mathbb R^{d})$ implies that for any given $\epsilon >0$, there exists $M_{0}\in \mathbb N$ such that 
$\| P_{M}^{\perp }b\|_{\BMO(\mathbb R^{d})}<\epsilon $ for all $M>M_{0}$.
%$P_{M}^{\perp }b\in \BMO(\mathbb R^{d})$ and, 
Also,
since $T_{P_{M}^{\perp }b}$ is a Calder\'on-Zygmund operator, we know by the classical theory that it is bounded from $L^{1}(\mathbb R^{d})$ into $L^{1,\infty }(\mathbb R^{d})$ with constant bounded by 
$\| P_{M}^{\perp }b\|_{\BMO(\mathbb R^{d})}$. 
With all this we can write
\begin{align*}
m( \{ x \in \mathbb R^{d} : |P^{\perp}_{M}T_bf(x)| > \lambda\} )
&=m( \{ x \in \mathbb R^{d} : |T_{P^{\perp}_{M}b}f(x)| > \lambda\} )
\\
&\lesssim \frac{1}{\lambda} \| P_{M}^{\perp }b\|_{\BMO(\mathbb R^{d})}
\| f\|_{L^{1}(\mathbb R^{d})}
\\
&\lesssim  \frac{\epsilon}{\lambda} \| f\|_{L^{1}(\mathbb R^{d})},
\end{align*}
which is the result we seek. 

Finally, we turn to the operator
$$T_b^\ast f(x) = \sum_{J \in \mathcal{D}}\inner{b}{\psi_{J}} \inner{f}{\psi_{J}} \varphi_J(x).$$ 
Our previous reasoning does not apply because, in general, $P^{\perp}_{M}T_b^\ast$ does not converge to zero. Namely, for $d=1$,  
$b=\psi_{[0,1]}$ and $\varphi=\chi_{[0,1]}$, we have that $T_b^\ast f=\inner{f}{\psi_{[0,1]}} \chi_{[0,1]}$, which is the operator we studied in example \ref{counterexample}. As we saw, $T_b^\ast $ is compact at the endpoint but $P_{M}T_b^\ast$ does not converge to $T_b^\ast $ in $L^{1,\infty }(\mathbb R)$.

However, by linearly, we still have $T_b^\ast = T_{P_M b}^\ast  + T_{P_M^\perp b}^\ast$. Now, $T_{P_M b}^{\ast}$ is of finite rank, and therefore
compact. Moreover, a similar argument as before shows that $m( \{ x \in \mathbb R^{d} : |T_{P_M^\perp b}^{\ast}f(x)| > \lambda\} )$ can be made smaller than $\epsilon/\lambda$ by choosing $M$ large. This proves compactness of $T_b^\ast $.\qed

\section*{Acknowledgements}

We extend our most sincere gratitude to
%We would like to thank 
\AA got Holand Olsen
%Olsen 
for arranging an enjoyable and productive stay in Harstad, Norway, where the seeds of this project first sprouted.  
We also acknowledge Fernando Cobos, Michael Lacey and Xavier Tolsa for helpful comments which improved the final version of our work. 
Last, the second author would also like to express his appreciation to 
\includegraphics[scale=0.34]{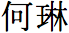}
%\hspace{-.25cm}
%\begin{CJK}{UTF8}{mj}
%%^^e4^^bb^^96
%^^e4^^bd^^95^^e7^^90^^b3
%\end{CJK}
%\hspace{-.25cm}
and George Janbing Lee for their invaluable help.

We also thank the support received from the Center for Mathematical Sciences at Lund University, Sweden, and by 
\includegraphics[scale=0.34]{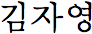}
%\hspace{-.25cm}
%\begin{CJK}{UTF8}{mj}
%^^ea^^b9^^80^^ec^^9e^^90^^ec^^98^^81
%%^^ec^^82^^ac
%%?
%%?
%\end{CJK}
%\hspace{-.25cm} 
%Research Park 
%(Foundation) 
in Sunnyvale, USA, the places where this research was 
successfully conducted.

\bibliographystyle{plain}

\begin{thebibliography}{10}

%\bibitem{AusHy}
%Pascal Auscher and Tuomas Hyt{\"o}nen.
%\newblock Orthonormal bases of regular wavelets in spaces of homogeneous type.
%\newblock {\em To appear}.
%
%\bibitem{Christ}
%M.~Christ.
%\newblock {\em Lectures on Singular Integral Operators}, volume~77 of {\em
%  {CBMS} Regional Conference Series in Mathematics}.
%\newblock {Conference Board of Mathematical Sciences; Amer. Math. Soc.},
%  Washington DC; Providence RI, 1990.
%  
%\bibitem{Cot}
%M.~Cotlar.
%\newblock {\em Continuity conditions for potential and {H}ilbert operators},
%  volume {F}asc. 2.
%\newblock {Departamento de Matematica, Facultad de Ciencias Exactas y
%  Naturales, Universidad Nacional de Buenos Aires}, 1959.
%
%
%\bibitem{David}
%G.~David.
%\newblock {\em Wavelets and singular integrals on curves and surfaces}.
%\newblock Number 1465 in Lecture notes in Mathematics. Springer-Verlag, Berlin,
%  1991.
%
%\bibitem{DenHan}
%Donggao Deng and Yongsheng Han.
%\newblock {\em {Harmonic {A}nalysis on {S}paces of {H}omegeneous type}}, volume 1966
%  of {\em Lecture notes in Mathematics}.
%\newblock Springer-Verlag, Berlin, 2009.


\bibitem{CobPer}
F.~Cobos, L.~E.~Persson.
\newblock {\em Real interpolation of compact operators between quasi-{B}anach spaces}.
\newblock Math. Scan., 82, 138-160, 1998.


\bibitem{Fab}
M.~Fabian, P.~Habala, V.~Montesinos, and V.~Zizler.
\newblock {\em Banach space theory. {T}he basis for linear and nonlinear
  analysis}.
\newblock {CMS} Books in Mathematics/Ouvrages de Math{\'e}matiques de la SMC.
  Springer, New York, 2011.


%\bibitem{Grafakos}
%L.~Grafakos.
%\newblock {\em Classical {F}ourier analysis}.
%\newblock Graduate Texts in Mathematics, 249, Second edition, Springer, New York, 2008.
%

\bibitem{HerWeiss}
E.~Hernandez and G.~Weiss.
\newblock {\em A first course on wavelets}.
\newblock CRC Press, Inc., Boca Raton, Florida, 1996.

%\bibitem{HyTa}
%Tomas Hyt{\"o}nen and Olli Tapiola.
%\newblock Almost {L}ipschitz-continuous wavelets in metric spaces via a new
%  randomization of dyadic cubes.
%\newblock {\em To appear}.

\bibitem{NTV}
F.~Nazarov, S.~Treil, and A.~Volberg. 
\newblock {Weak Type Estimates and Cotlar Inequalities for Calder\'on-ZygmundOperatorson Nonhomogeneous Spaces}.
\newblock {\em International Mathematics Research Notices},  9, 463--487, 1998.

\bibitem{LTW}
M.~Lacey, E.~Terwilleger, and B.~Wick.
\newblock {Remarks on Product {V}{M}{O}}.
\newblock {\em Proc. Amer. Math. Soc}, 2:465--474, 2006.

%\bibitem{journe}
%J.-L.~Journ\' e.
%\newblock {\em {Calder\' on-Zygmund {O}perators, {P}seudo-{D}ifferential {O}perators and the {C}auchy {I}ntegral of {C}alder\' on} },
%volume 994 of {\em Lecture notes in Mathematics}.
%\newblock Springer-Verlag, Berlin, 1983.
%
%\bibitem{Kra}
%M.~A. Krasnoselski.
%\newblock On a theorem of {M}. {R}iesz.
%\newblock {\em Dokl. Akad. Nauk}, 131:246--248, 1959.
%
%
%\bibitem{nazarov_treil_volberg1998}
%N.~Nazarov, S.~Treil, and A.~Volberg.
%\newblock { {Weak type estimates and {C}otlar inequalities for {C}alder\' on-{Z}ygmund operators on nonhomogeneous spaces}}.
%\newblock {\em Int. Math. Res. Not. IMRN}, 9:463--487, 1998.
%
%
%\bibitem{Pajot}
%Herv\'e Pajot.
%\newblock {\em {Analytic capacity, rectifiability, {M}enger curvature and the
%  Cauchy Integral}}, volume 1799 of {\em Lecture notes in Mathematics}.
%\newblock Springer-Verlag, Berlin, 2002.

\bibitem{PPV}
K.~M. Perfekt, S.~Pott, and P.~Villarroya.
\newblock Endpoint compactness of singular integrals and perturbations of the Cauchy Integral.
\newblock {\em Submitted.}


\bibitem{Roch}
R.~Rochberg.
\newblock Toeplitz and {H}ankel operators on the {P}aley-{W}iener space.
\newblock {\em Integral Equations Operator Theory}, 10(2):187--235, 1987.

\bibitem{Runde}
V.~Runde.
\newblock A new and simple proof of {S}chauder's theorem,
\newblock {http://arxiv.org/pdf/1010.1298.pdf}


%\bibitem{ST}
%E.~M. Stein.
%\newblock {\em Harmonic Analysis: real-variable methods, orthogonality and
%  oscilatory integrals}.
%\newblock Princeton Univ. Press, 1993.

\bibitem{V1} P.~Villarroya.
\newblock {A characterization of compactness for singular integrals}. 
\newblock {\em To appear in Journal de Math\'ematiques Purees et Appliqu\'ees}



\bibitem{V2}
P.~Villarroya.
\newblock {A global $Tb$ theorem for boundedness and compactness}.
\newblock {\em Submitted}.

\bibitem{TLec}
C.~Thiele.
\newblock {\em Wave packet analysis}, volume 105 of {\em {CBMS} Regional
  Conference Series in Mathematics}.
\newblock Amer. Math. Soc., Providence RI, 2006.

%\bibitem{Tolsa}
%X.~Tolsa.
%\newblock {$L^{2}$ boundedness of the {C}auchy integral operator for continuous
%  measure}.
%\newblock {\em Duke Math. J.}, 98(2):269--304, 1999.




\end{thebibliography}

\end{document}